\documentclass[bj]{imsart}

\RequirePackage[OT1]{fontenc}
\RequirePackage{amsthm,amsmath,latexsym}
\RequirePackage{amsthm,amsmath,latexsym}
\RequirePackage[numbers]{natbib}
\RequirePackage[colorlinks,citecolor=blue,urlcolor=blue]{hyperref}

\usepackage[dvips]{graphicx}
\usepackage{amsfonts}
\usepackage{multirow}

\startlocaldefs
\numberwithin{equation}{section}
\theoremstyle{plain}

\newtheorem{definition}{Definition}[section]
\newtheorem{proposition}[definition]{Proposition}
\newtheorem{theorem}[definition]{Theorem}
\newtheorem{corollary}[definition]{Corollary}
\newtheorem{lemma}[definition]{Lemma}

\newtheorem{example}{Example}
\newtheorem{remark}{Remark}

\endlocaldefs

\def\argmax{\mathop{\rm argmax}}
\def\argmin{\mathop{\rm argmin}}
\def\Ker{\mathop{\rm Ker}}
\def\dim{\mathop{\rm dim}}

\newcommand{\bR}{\mathbb{R}}

\newcommand{\cX}{{\cal X}}
\newcommand{\cY}{{\cal Y}}

\def\cY{{\cal Y}}

\def\rank{{\rm rank}}

\makeatletter
\newcommand{\lleq}{\mathrel{\mathpalette\gl@align<}}
\newcommand{\ggeq}{\mathrel{\mathpalette\gl@align>}}
\newcommand{\gl@align}[2]{
\vbox{\baselineskip\z@skip\lineskip\z@
\ialign{$\m@th#1\hfil##\hfil$\crcr#2\crcrV_{\vec{\theta}}_{{}_{(=)}}\crcr}}}
\makeatother

\def\Label#1{\label{#1}\ [\ \text{#1}\ ]\ }
\def\Label{\label}

\newenvironment{proofof}[1]{\vspace*{5mm} \par \noindent
{\it Proof of #1:\hspace{2mm}}}{\endproof
\hfill$\Box$ 
}

\newenvironment{proofsof}[1]{\vspace*{5mm} \par \noindent
         \quad{\it Proofs of #1:\hspace{2mm}}}{\endproof
\hfill$\Box$ 
}

\def\PF#1{\noindent{\sf #1}:\quad}

\newcommand{\mh}[1]{#1}

\begin{document}

\begin{frontmatter}
\title{Information Geometry Approach to Parameter Estimation in Hidden Markov Model}
\runtitle{Information Geometry Approach in Hidden Markov Model}

\begin{aug}
\author[A]{Masahito Hayashi}\ead[label=e1]{hayashi@sustech.edu.cn},\ead[label=e2]{masahito@math.nagoya-u.ac.jp}
\address[A]{
Shenzhen Institute for Quantum Science and Engineering, Southern University of Science and Technology,
Shenzhen, 518055, 
China. \\
Guangdong Provincial Key Laboratory of Quantum Science and Engineering,
Southern University of Science and Technology, Shenzhen, 518055, 
China. \\
Shenzhen Key Laboratory of Quantum Science and Engineering, Southern
University of Science and Technology, Shenzhen, 518055, 
China. \\
Graduate School of Mathematics, Nagoya University, Nagoya, 464-8602, Japan. \\
Center for Advanced Intelligence Project, RIKEN, Tokyo, 103-0027, Japan.\\
Centre for Quantum Technologies, National University of Singapore, 117543, Singapore.\\
Center for Quantum Computing, Peng Cheng Laboratory, Shenzhen, 518066, China.
\printead{e1,e2}}
\end{aug}

\begin{abstract}
We consider the estimation of 
the transition matrix of a hidden Markovian process by using information geometry with respect to transition matrices.
In this paper, only the histogram of $k$-memory data is used for the estimation.
To establish our method, we focus on a partial observation model with the Markovian process and we 
propose an efficient estimator whose asymptotic estimation error 
is given as the inverse of projective Fisher information of transition matrices.
This estimator is applied to the estimation of the transition matrix of the hidden Markovian process.
In this application, we carefully discuss the equivalence problem for hidden Markovian process on the tangent space.
\end{abstract}


\end{frontmatter}

\section{Introduction}\Label{s1}
Information geometry has been established by Amari and Nagaoka \cite{AN}
as a very powerful method for statistical inference.
Recently, this approach was applied to the estimation of the transition matrix of the Markovian process \cite{HW-est}.
When we observe the $n$ sequential data subject to a certain Markovian process,
if we focus on the geometry based on the probability distribution of the $n$ data, 
the geometrical structure changes according to the increase of the number $n$.
To resolve this problem, 
the paper \cite{HW-est} focuses on information geometry of transition matrices given by 
Nakagawa and Kanaya \cite{NK} and Nagaoka \cite{HN}. 
This is because this geometric structure depends only on the transition matrices, and 
does not change as the number $n$ increases.
Based on this property, the paper \cite{HW-est} introduced the curved exponential family of transition matrices, and derived the Cram\'{e}r-Rao inequality for the family, 
which shows the optimality of the inverse of Fisher information matrix for the transition matrices.

\begin{figure}[htbp]
\begin{center}
\scalebox{1}{\includegraphics[scale=0.5]{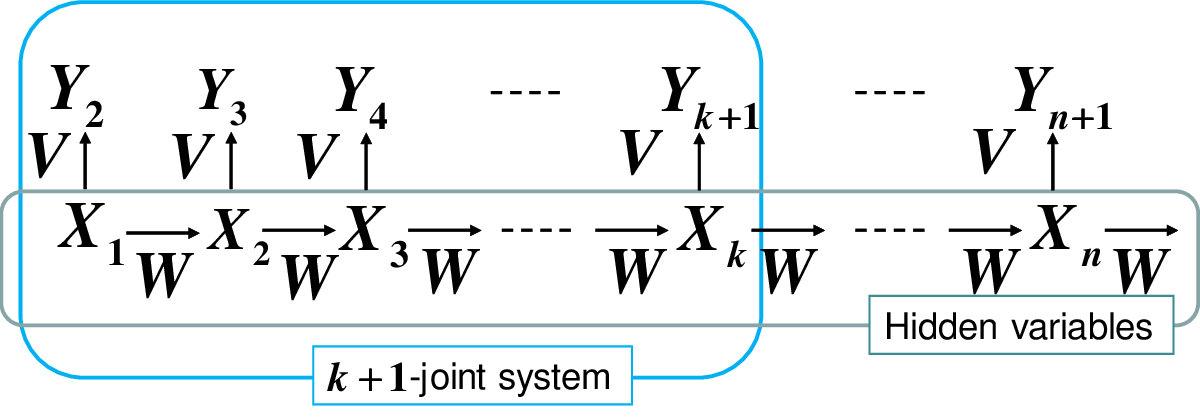}}
\end{center}
\caption{Pair of transition matrices: 
The transition matrix $W$ determines the Markovian process on the set ${\cal X}$ of hidden states.
The transition matrix $V$ determines the observed variable $Y$
with the condition on the hidden variable $X$.
If $k$ is sufficiently large, the joint distribution on $Y_1, \ldots, Y_{k+1}$ uniquely determines
the transition matrices $W$ and $V$.
The partial observation model is applied to 
the $k+1$ joint system on $X_{1},Y_2,\ldots, X_{k+1},Y_{k+2}$.}
\Label{2model}
\end{figure}%

On the other hand, 
the hidden Markov model (HMM) is an important
statistical model in many fields including Bioinformatics \cite{Durbin}
Econometrics \cite{Kim} and Population genetics \cite{Felsenstein}
(See also the recent overview \cite{Capp}.).
Some of preceding studies \cite{Amari-em,FN} to estimate the transition matrix of the hidden Markov process have already employed information geometry by using an em-algorithm
when the transition matrix is unknown and the state spaces is finite.
However, their em-algorithm is based on the geometry of probability distributions, which changes according to the increase of the number $n$ of observations. 
So, their estimation processes become complicated when $n$ is large.
Hence, it is hard to evaluate the asymptotic behavior of the estimation error, rigorously.

In this paper, we apply the information geometry of transition matrices to the estimation of 
the transition matrix of the hidden Markovian process with finite state space.
Since we need to estimate the hidden structure from the observed value,
we apply the em-algorithm based on the geometry of transition matrices.
\mh{Usually, the joint distribution for $Y_i,Y_{i+1}$ cannot uniquely determine the transition matrix to characterize the hidden Markov process.
However, if the positive integer $k$ is sufficiently large and the system is stationary,
the joint distribution for $Y_i,\ldots, Y_{i+k}$, i.e., the $k$-memory transition matrix for $Y$, 
uniquely determines the transition matrix to characterize the hidden Markov chain.
In this paper, exploiting this relation, we propose our method to estimate the transition matrix to characterize the hidden Markov process.
For this aim, we formulate a partial observation model of the Markovian process, which is composed of 
observed variables and unobserved variables, and establish the em-algorithm based on this model.
Then, we apply the partial observation model to the $k+1$ joint system on $X_{i-1},Y_i,\ldots, X_{i+k},Y_{i+k+1}$, where
the observed variables are given as random variables determined by $Y_i,\ldots, Y_{i+k+1}$.}

Next, using the transition matrix version of the projective Fisher information,
we evaluate the asymptotic error in the partial observation model of the Markovian process under a certain regularity condition.
To apply the partial observation model of the Markovian process
to the observed $k$-memory joint distribution,
we formulate an exponential family of $k$-memory transition matrices.
As the merit of our method, the calculation complexity of the em-algorithm does not depend on the number $n$ of observations
in this application because the geometrical structure does not depend on $n$.
As another result of the partial observation model of the Markovian process, 
we propose an efficient method for $\chi^2$-test to verify the validity of the given model, 
and show that it works well.
Since it checks the validity of a given model, we can employ it to choose a model among plural models.

We have another difficulty in the hidden Markovian process.
There is ambiguity for the transition matrix to express the hidden Markovian process. 
That is, there is a possibility that two different transition matrices express
the same hidden Markovian process. 
This problem is called the equivalence problem and was solved by Ito, Kobayashi, and Amari \cite{IAK}.
The asymptotic error is characterized by the local geometrical structure,
and the equivalence problem with local structure was discussed in another paper \cite{LEH},
which is needed to make parametrization without duplication.
Employing the parametrization given in \cite{LEH}, 
we apply our partial observation model of the Markovian process to the estimation of 
the transition matrix of the hidden Markovian process.

The remaining of this paper is organized as follows.
Section \ref{s2} gives a brief summary of the obtained results, which is crucial for understanding the structure of this paper.
As a preparation, Section \ref{s3} reviews the fundamental facts for an exponential family of transition matrices and its properties.
Section \ref{s4} discusses the partial observation model for the Markovian process, in which, we can observe a part of random variables.
Section \ref{s4-5} applies the model to the case of $k$-memory transition matrices.
To apply our result to the hidden Markovian process given in Fig. \ref{2model},
Section \ref{s8} gives a proper parametrization
based on the results of the paper \cite{LEH} for an exponential family of pairs of transition matrices.
Section \ref{s6B} applies the result of Section \ref{s4} to the estimation of the transition matrix of the hidden Markovian process.

\section{Summary of results}\Label{s2}
\subsection{Estimation in partial observation model}\Label{s2-1}
As the first part of our main result, in Section \ref{s4}, we address 
the partial observation model in a general Markovian process,
and discuss the estimation of the parameter to identify the 
Markovian process in a parametric family of Markovian processes.
Since the partial observation model is a general model containing 
hidden Markovian process,
this discussion works as a crucial preparation of estimation of hidden Markovian process.
\mh{In Sections 5, 6, and 7,
the partial observation model will be applied 
to the $k$ joint distribution on $X_i,Y_{i+1},\ldots, X_{i+k},Y_{i+k+1}$.}
We consider a large exponential family of 
transition matrices $W_{\vec{\theta}_1,\vec{\theta}_{2,3}}$
generated by 
$l_1+l_2+l_3$ independent variables $g_{1}, \ldots, g_{l_1}, g_{l_1+1}, \ldots, g_{l_1+l_2+l_3} $.
In this setting,
our observation is limited to the $l_1+l_2$ variables $g_1, \ldots, g_{l_1+l_2} $,
and the true transition matrix is assumed to belong to 
an exponential family of transition matrices $W_{0,\vec{\theta}_{2,3}}$.
\mh{In the above application, the variables $g_1, \ldots, g_{l_1+l_2} $ are functions of 
$Y_{i+1},\ldots, Y_{i+k+1}$.}
The definition of the exponential family of 
transition matrices generated by variables is given in Section \ref{s3-2}.

In this setting, we discuss how to estimate the parameter $\vec{\theta}$ from the average $\vec{Y}_{1,2}^n$ of 
$n$ observed data of the variables $g_1, \ldots, g_{l_1+l_2} $.
For this aim, as in Section \ref{s3-2},
we introduce the expectation parameter $\vec{\eta}_{1,2},\vec{\eta}_3$ corresponding 
the variables $g_{1}, \ldots, g_{l_1+l_2+l_3} $.
Using Legendre transform, we define the conversion from 
the expectation parameter $\vec{\eta}_{1,2},\vec{\eta}_3$ to
the natural parameter $\vec{\theta}(\vec{\eta}_{1,2},\vec{\eta}_3) $, 
whose detail calculation formula is given in \eqref{eq1-28-2}.
Then, we propose our estimator as
\begin{align}
\hat{\theta}_{2,3}(\vec{Y}_{1,2}^n) :=
\argmin_{\vec{\theta}_{2,3}' \in \Theta_{2,3}}
\min_{\vec{\eta}_3' \in {\cal V}_3}
D\big(W_{\vec{\theta}(\vec{Y}_{1,2}^n,\vec{\eta}_3')}\big\|W_{0,\vec{\theta}_{2,3}'}\big),
\Label{eq6-9-1R}
\end{align}
where ${\cal V}_3 $ is the $l_3$-dimensional real vector space and 
$D(W_{\vec{\theta}_1,\vec{\theta}_{2,3}}\|W_{\vec{\theta}_1',\vec{\theta}_{2,3}'})$
is the divergence between two transition matrices
$W_{\vec{\theta}_1,\vec{\theta}_{2,3}}$ and $W_{\vec{\theta}_1',\vec{\theta}_{2,3}'}$,
whose detail calculation formula is given in Section \ref{s3-4}.
As explained in Section \ref{s4}, this estimator can be calculated by using the em-algorithm with respect to 
the divergence for transition matrices.
Then, we show the following theorem.
\begin{theorem}\Label{PT1}
Under regularity conditions {\bf C1}, {\bf C2}, and {\bf C3}
given in Section \ref{s4}, 
the distribution of the estimation error of the estimator $\hat{\theta}_{2,3}(\vec{Y}_{1,2}^n) $ 
converges to the zero-mean Gaussian distribution
whose variance is the inverse of the projective Fisher information, whose definition will be given in 
\eqref{FTE}.
\end{theorem}

In the above theorem,
the estimation error means 
the random variable $\hat{\theta}_{2,3}(\vec{Y}_{1,2}^n) - \vec{\theta}_{2,3;o}$,
where $\vec{\theta}_{2,3;o}$ is the true natural parameter.
Next, we consider how to select a suitable model among several models.
In the case of hidden Markov processes, a submodel is  given as a set of singular points.
Hence, it is quite hard to apply conventional model selection methods.
Fortunately, the $\chi^2$-test can be formulated as follows.

\begin{theorem}\Label{T26}
We assume that $\Theta_{2,3}$ is an open set, 
any element of  $\Theta_{2,3}$ satisfies the same condition as Theorem \ref{PT1},
and the true parameter belongs to $\Theta_{2,3}$.
The random variable 
$2n \min_{\vec{\theta}_{2,3}' \in \Theta_{2,3}}
\min_{\vec{\eta}_3' \in {\cal V}_3}
D(W_{\vec{\theta}(\vec{Y}_{1,2}^n,\vec{\eta}_3')}\|W_{0,\vec{\theta}_{2,3}'})$
is subject to $\chi^2$-distribution with degree $l_1-l[\Theta_{2,3}]$,
where $l[\Theta_{2,3}]$ 
is an integer defined in Theorem \ref{6-13-thB}.
\end{theorem}

Hence, when the number $n$ of observations is sufficiently large,
the random variable given in Theorem \ref{T26} can be used for the $\chi^2$-test.
Thus, a suitable model can  be selected among several models 
when we fixed a certain significance level $\alpha>0$.
We choose the model with the minimum dimension 
among models passing this $\chi^2$-test with the significance level $\alpha$.
As explained in Section \ref{s4}, 
the minimum in \eqref{eq6-9-1R} and Theorem \ref{T26} can be calculated by 
the em-algorithm, and the calculation complexity does not depend on $n$.

\subsection{Estimation in hidden Markovian model}\Label{s2-2}
Next, we address the estimation of the transition matrix of the hidden Markovian process given in Fig. \ref{2model},
where ${\cal Y}$ is the observed finite state system and ${\cal X}$ is the hidden finite state system. 
In this paper, to clarify the correspondence between the finite state system and the random variable, 
we denote the finite state space by ${\cal X}$ when it corresponds to the variable $X$.
As explained in Section \ref{s1}, the hidden Markov process  is modeled by a pair of
transition matrices $(W,V)$.
For this aim, after reviewing the result of the paper \cite{LEH},
we parametrize the hidden Markov process
as an exponential family of pairs of transition matrices $(W_{\vec{\theta}},V_{\vec{\theta}})$
with generators $g_{l_1+1}, \ldots, g_{l_1+l_2+l_3}$.
Then, we estimate the parameter $\vec{\theta}$ under the above family
from $\bar{n}-1$ observed data $Y_1, \ldots, Y_{\bar{n}-1}$.
If we employ several functions of $Y_1, \ldots, Y_{\bar{n}-1}$ as our estimator of $\vec{\theta}$,
the calculation complexity increases in some order as $\bar{n}$ increases.
To avoid this problem, 
we impose the following constraints.
We set $\bar{n}=n+k$ and allow to use only the averages of $k$-input functions.
That is, we prepare functions $g_1, \ldots, g_{l_1+l_2}$ with $k$ inputs
and use only their averages
$Y_i^n:=\sum_{j=0}^{n-1} \frac{1}{n}g_i(Y_{j+1}, \ldots,Y_{j+k}) $ with $i=1, \ldots, l_1+l_2$.
Here, $g_{l_1+1}, \ldots, g_{l_1+l_2}$ are observed generators across $k$ observed systems.
In this case, the estimate is given as a function 
$\vec{Y}^n:=(Y_1^n,\ldots, Y_{l_1+l_2}^n)\mapsto \hat{\theta}$.

We focus on the Markovian process with $k-1$-memory,
and apply the discussion in Section \ref{s2-1} to the $l_1+l_2+l_3$ variables $g_1, \ldots, g_{l_1+l_2+l_3} $ 
given in the above way.
Then, the exponential family of pairs of transition matrices $(W_{\vec{\theta}},V_{\vec{\theta}})$
is parametrized by an $l_2+l_3$-dimensional parameter space $\Theta_{2,3}$, and
we have the estimate $\hat{\theta}_{2,3}(\vec{Y}^n)  $ given in \eqref{eq6-9-1R}.
In this application, we obtain the following result. 

\begin{theorem}\Label{PT2}
We fix $k$ 
as an integer greater than a certain threshold 
\mh{(the concrete value will be given in Condition {\bf F3} after Theorem \ref{T3-29}.)}, 
and 
\mh{we consider the limit as $n$ tends to infinity.}
\mh{If the parameter space $\Theta_{2,3}$ satisfies some regularity conditions
and the true pair of transition matrices belongs to 
$\{(W_{\vec{\theta}},V_{\vec{\theta}})\}_{\vec{\theta}\in \Theta_{2,3}}$,
then} the projective Fisher information 
is invertible and 
the distribution of the estimation error of the estimator $\hat{\theta}_{2,3}(\vec{Y}^n)  $ converges to the zero-mean Gaussian distribution
whose variance is the inverse of the projective Fisher information.
(For the definition of the estimation error, see the sentence after Theorem \ref{PT1}.)
\end{theorem}

Since 
the estimator of Theorem \ref{PT1} is calculated by the em-algorithm,
the calculation complexity of our estimator from the averages
$Y_i^n:=\sum_{j=0}^{n-1} \frac{1}{n}g_i(Y_{j+1}, \ldots,Y_{j+k}) $ with $i=1, \ldots, l_1+l_2$
depends only on the integer $k$ and does not depend on the other integer $n$.
This is because the maximization problem to obtain the estimate from 
$Y_i^n$
is written as a maximization problem in a vector space whose dimension 
depends only on $k$ and the sizes of state spaces of ${\cal X}$ and ${\cal Y}$,
and is independent of $n$.
As the integer $k$ is fixed, the calculation complexity is $O(1)$ once we obtain
the averages. 
Although the calculation complexity of $Y_i^n$ is $O(n)$, its space complexity is $O(\log n)$,
which is an advantage over existing methods \cite{Amari-em,FN,Baum,KGMS}
\mh{as explained in the 4th paragraph of Section \ref{s9}.}

\section{Exponential family of transition matrices}\Label{s3}
\subsection{Preparation}
The aim of this paper is to handle hidden Markov process in the terms of exponential family of transition matrices.
To achieve this aim, we need 
several definitions for transition matrices on the finite state space ${\cal X}$, 
which are introduced as follows.
A non-negative matrix $W$ is called {\it irreducible} 
when for each $x,x'\in \cX$, there exists a natural number $n$ such that $W^n(x|x')>0$ \cite{MU}.
The irreducibility depends only on the support $\cX^2_W:=\{(x,x') \in \cX^2| W(x|x')>0\}$
for a non-negative matrix $W$ over $\cX$.
In the irreducible case, the average $ \sum_{i=1}^n \frac{1}{n} W^i P$
converges to the stationary distribution $P_{W}$ for any initial distribution $P$ on ${\cal X}$
as $n$ goes to infinity \cite{DZ,kemeny-snell-book}. 
An irreducible matrix $W$ is called {\it ergodic} when 
there are no input $x'$ and no integer $n'$
such that $W^{n} (x'|x')=0$ unless 
$n$ is divisible by $n'$ \cite{MU} \footnote{The definition of the ergodicity depends on
the research area. 
For example, in the quantum information,
the ergodicity is defined as the property of the convergence of
the average $ \sum_{i=1}^n \frac{1}{n} W^i P$ to the unique stationary distribution.
In this area, the property of the output distribution is called mixing property \cite{BG,Burgarth,Yoshida}.}.
It is known that 
the output distribution $W^n P$ converges to the stationary distribution of $W$
if and only if the transition matrix $W$ is ergodic \cite{kemeny-snell-book,DZ,MU}.
In this section, we treat only irreducible transition matrices and 
the class of ergodic transition matrices will be discussed in Section \ref{s4-5-1}.

To handle the transition matrix by using the joint distribution on ${\cal X}^2$,
we introduce the stationary condition for a distribution $P$ on ${\cal X}^2$ as
\begin{align}
\sum_{x' \in {\cal X}} P(x', x)
=\sum_{x' \in {\cal X}} P(x,x').\Label{symB}
\end{align}
For a transition matrix $W$ on ${\cal X}$,
we introduce 
the joint stationary distribution $P_W^2$ as $P_W^2(x,x') := W(x|x') P_W(x')$
and
the set of transition matrices ${\cal W}_{\cX,W}$
as\footnote{${\cal X}^2_{V}$
is defined by the same way as
${\cal X}^2_{W}$ with replacing the transition matrix $W$ by $V$.}
\begin{align}
{\cal W}_{\cX,W} := \{V| V \hbox{ is a transition matrix and }
{\cal X}^2_{W}={\cal X}^2_V \}.
\end{align}
Then, we have the following lemma.

\begin{lemma}\Label{STR3W}
The set of distributions $P_W^2$ equals that the set of 
distributions on ${\cal X}^2_W$ satisfying the stationary condition \eqref{symB}.
More precisely, 
\begin{align*}
\{ P_{W'}^2 | W' \in {\cal W}_{\cX,W}\}
=
\left\{ P\left| 
\begin{array}{l}
P \hbox{ is a distribution with support } {\cal X}^2_W \\
\hbox{to satisfy the stationary condition \eqref{symB}.}
\end{array}
\right.\right\}
\end{align*}
\end{lemma}

\begin{proof}
Since the joint distribution $P_W^2$ satisfies the stationary condition \eqref{symB} for any 
$W' \in {\cal W}_{\cX,W}$, we have the relation $\subset$.
For a joint distribution $P$ on ${\cal X}^2$ satisfying the stationary condition \eqref{symB},
we define the conditional distribution $W$ as $W(x|x'):=\frac{P(x,x')}{P_X(x')}$, where
$P_X(x'):=\sum_{x''}P(x'',x')$.
Due to the stationary condition \eqref{symB},
$P_X$ is the stationary distribution.
Hence, we obtain the opposite relation $\supset$.
\end{proof}

\subsection{Linear independence}\Label{s3-1}
Exponential family of transition matrices is generated by the variables defined as functions on 
$\cX^2_W $. 
The linear independence of generators is a key concept for the exponential family
because we need to avoid duplicate parametrization of transition matrices.
Hence, before introducing an exponential family, 
we prepare the following lemma.

\begin{proposition}[\protect{\cite[Lemma 3.1]{HW-est}}]\Label{L1}
Consider an irreducible transition matrix $W$ over $\cX$
and a real-valued function $g$ on $\cX \times \cX$.
Define $\phi(\theta)$ as the logarithm of the Perron-Frobenius eigenvalue of the matrix\footnote{For the 
Perron-Frobenius eigenvalue and Perron-Frobenius eigenvector, see Appendix \ref{PFTH}.}:
$
\overline{W}_{\theta}(x|x'):= W(x|x') e^{\theta g(x,x')}$.
Then, the function $\phi(\theta)$ is convex.
Further, the following conditions for the function $g$ are equivalent.
\begin{description}
\item[(1)] 
No real-valued function $f$ on $\cX$ satisfies that 
$g(x,x')=f(x)-f(x')+c$ for any $(x,x')\in \cX^2_W$ with a constant $c \in \mathbb{R}$.
\item[(2)]
The function $\phi(\theta)$ is strictly convex, i.e., 
$\frac{d^2 \phi}{d \theta^2}(\theta)>0$ for any $\theta$.
\item[(3)]
$\frac{d^2 \phi}{d \theta^2}(\theta)|_{\theta=0}>0$.
\end{description}
\end{proposition}

Given two distinct transition matrices $W$ and $V$,
we assume that $\cX^2_W \subset \cX^2_V$ and $\cX^2_W$ is irreducible.
Then, we denote the logarithm of 
the Perron-Frobenius eigenvalue of the matrix
$W(x|x')^{1+s}V(x|x')^{-s}$ by $\varphi_{W,V}(1+s)$ under the condition given below.
Proposition \ref{L1} guarantees the differentiability of $\varphi_{W,V}(1+s)$.
So, we define the divergence $D({W} \| V)$ as
$
D({W} \| V):= 
\frac{d \varphi_{W,V}}{d s}(1)$. 

\subsection{Exponential family}\Label{s3-2}
To define an exponential family for transition matrices,
we focus on an irreducible transition matrix $W(x|x')$ from $\cX$ to $\cX$.
A set of real-valued functions $\{g_j\}$ on $\cX \times \cX$
is called {\it linearly independent} under the transition matrix $W(x|x')$
when any linear non-zero combination of $\{g_j\}$ satisfies 
the condition (1), (2) or (3) for the function $g$ given in Proposition \ref{L1}.
For $\vec{\theta}=(\theta^1, \ldots,\theta^l)$
and linearly independent functions $\{g_j\}$,
we define the matrix ${W}_{\vec{\theta}}(x|x')$ from $\cX$ to $\cX$ in the following way.
\begin{align}
\overline{W}_{\vec{\theta}}(x|x'):= W(x|x') e^{\sum_{j=1}^l  \theta^j g_j(x,x')}.
\Label{5-6}
\end{align}
Using the Perron-Frobenius eigenvalue $\lambda_{\vec{\theta}}$ of $\overline{W}_{\vec{\theta}}$,
we define the potential function $\phi(\vec{\theta}):=\log \lambda_{\vec{\theta}}$. 
In the following, we introduce geometrical structure based on the potential function $\phi(\vec{\theta})$, 
i.e., we employ the information geometry induced by the potential function $\phi(\vec{\theta})$
because the potential function $\phi(\vec{\theta})$ yields
all the statistical characteristics even in a Markovian process \cite{HW14-2}.

Note that, since the value $\sum_{x}\overline{W}_{\vec{\theta}}(x|x')$ generally depends on $x'$, 
we cannot make a transition matrix by simply multiplying a constant with the matrix $\overline{W}_{\vec{\theta}}$.
For deeper understanding, 
we employ the linear space ${\cal V}_{{\cal X}}:=\{
v=(v_x)_{x \in {\cal X}}| v_x \in \mathbb{R} \}$.
The inner product between $v,v'\in {\cal V}_{{\cal X}}$
is given as $\langle v|v'\rangle:= \sum_{x \in {\cal X}}v_x {v_x'}$. 
We identify the vector $v'$ with the notation $|v'\rangle$.
The functional over ${\cal V}_{{\cal X}}$ given as
$v'\mapsto\langle v|v'\rangle$
is denoted by $\langle v|$.
Hence, $|v'\rangle\langle v|$ expresses the linear map on 
${\cal V}_{{\cal X}}$: $v''\mapsto |v'\rangle\langle v|v''\rangle$.
A transition matrix $W$ on ${\cal X}$ can be regarded as
a linear map on ${\cal V}_{{\cal X}}$:
$(W v)_x:=\sum_{x' \in {\cal X}}W(x|x')v_{x'} $.

To make a transition matrix from the matrix $\overline{W}_{\vec{\theta}}$, we recall that
a non-negative matrix $U$ from $\cX$ to $\cX$ is a transition matrix
if and only if the vector $u_{{\cal X}}:=(1, \ldots,1)^T \in {\cal V}_{{\cal X}}$ 
is an eigenvector of the transpose $U^T$
and the eigenvalue is $1$.
In order to resolve this problem, we focus on the structure of the matrix $\overline{W}_{\vec{\theta}}$.
We denote the Perron-Frobenius eigenvectors
of $\overline{W}_{\vec{\theta}}$ and its transpose $\overline{W}_{\vec{\theta}}^T$
by $\overline{P}^2_{\vec{\theta}}$ and $\overline{P}^3_{\vec{\theta}}$.
Then, 
similar to \cite[(16)]{NK} \cite[(2)]{HN},
we define the matrix ${W}_{\vec{\theta}}(x|x')$ as
\begin{align}
{W}_{\vec{\theta}}(x|x'):= \lambda_{\vec{\theta}}^{-1} \overline{P}^3_{\vec{\theta}}(x)
\overline{W}_{\vec{\theta}}(x|x')\overline{P}^3_{\vec{\theta}}(x')^{-1}.\Label{12-26-1}
\end{align}
The matrix ${W}_{\vec{\theta}}(x|x')$ is 
a transition matrix because the vector $(1, \ldots,1)^T$ is an eigenvector of the transpose $W_{\vec{\theta}}^T$
and its Perron-Frobenius eigenvalue is $1$.
The stationary distribution $P_{{W}_{\vec{\theta}}}$ of the given 
transition matrix ${W}_{\vec{\theta}}$
is the Perron-Frobenius normalized eigenvector of the transition matrix ${W}_{\vec{\theta}}$,
which is given as 
$
P_{{W}_{\vec{\theta}}}(x)
=
\frac{\overline{P}^3_{\vec{\theta}}(x)\overline{P}^2_{\vec{\theta}}(x)}{
\sum_{x''}\overline{P}^3_{\vec{\theta}}(x'')\overline{P}^2_{\vec{\theta}}(x'')}
$\cite[(4.3)]{HW-est}.
In the following, we call the family of transition matrices 
${\cal E}:=\{ {W}_{\vec{\theta}} \}$ an {\it exponential family} of transition matrices 
generated by $W$ with the generator $\{g_1,\ldots,g_l\}$.

Since the generator $\{g_1,\ldots,g_l\}$ is linearly independent,
due to Proposition \ref{L1},\par\noindent
$\sum_{i,j}c^i c^j\frac{\partial^2 \phi}{\partial \theta^i\partial \theta^j}=
\frac{d^2 \phi(\vec{c}t)}{d t^2}$ 
is strictly positive definite for an arbitrary non-zero vector $\vec{c}=(c^1, \ldots, c^l)$.
That is, the Hesse matrix 
$\mathsf{H}_{\vec{\theta}} [\phi]=[\frac{\partial^2 \phi}{\partial \theta^i\partial \theta^j}]_{i,j}$ 
is strictly positive.
Using the potential function $\phi(\theta)$,
we discuss several concepts 
for transition matrices based on Proposition \ref{L1}, formally.
While the parameter $(\theta^1,\ldots,\theta^l)$ is called 
the {\it natural parameter},
we introduce the {\it expectation parameter}
$\eta_j(\vec{\theta}):= 
\frac{\partial \phi}{\partial\theta^j}(\vec{\theta})
=\sum_{x,x'} g_j(x, x')
P_{W_{\vec{\theta}}}^2(x,x')$.
For $\vec{\eta}=(\eta_1, \ldots,\eta_l)$,
we define $\theta^1(\vec{\eta}), \ldots, \theta^l(\vec{\eta})$
as $\eta_j(\theta^1(\vec{\eta}), \ldots, \theta^l(\vec{\eta}))=\eta_j$.

Here, to avoid the duplicate parametrization, we introduce the quotient space.
For a given transition matrix $W$, 
we consider a two-input function $g(x,x')$ defined on the subset ${\cal X}^2_W$, and denote the linear space of such two-input functions
by ${\cal G}({\cal X}^2_W)$.
We define the linear subspace ${\cal N}({\cal X}^2_W)$ of
the space ${\cal G}({\cal X}^2_W)$
as the set of functions $f(x)-f(x')+c$.
Then, we discuss the quotient space ${\cal G}({\cal X}^2_W)/{\cal N}({\cal X}^2_W)$, and the equivalent class containing $g$ is written as $[g]$\footnote{For the definition of 
quotient space, see the textbook \cite{Halmos}.}.

\begin{proposition}[\protect{\cite[Lemma 4.1]{HW-est}}]\Label{L1-14-2}
The following conditions are equivalent for the generator $\{g_j\}$
and the transition matrix $W$.
\begin{description}
\item[(1)]
The set of functions $\{g_j\}$ are linearly independent in the 
quotient space ${\cal G}({\cal X}^2_W)/{\cal N}({\cal X}^2_W)$.

\item[(2)]
The map $\vec{\theta} \to \vec{\eta}(\vec{\theta})$
is one-to-one.

\item[(3)]
The Hesse matrix $\mathsf{H}_{\vec{\theta}} [\phi]$ is strictly positive definite for any $\vec{\theta}$,
which implies the strict convexity of the potential function $\phi(\vec{\theta})$.

\item[(4)]
The Hesse matrix $\mathsf{H}_{\vec{\theta}} [\phi]|_{\vec{\theta}=0}$ is strictly positive.

\item[(5)]
The parametrization 
$\vec{\theta}\mapsto W_{\vec{\theta}}$ is faithful for any $\vec{\theta}$.
\end{description}
\end{proposition}

For linearly independent functions, we have the following lemma.

\begin{lemma}\Label{STRW}
The following conditions for linearly independent function $g_1, \ldots, g_l$
in 
${\cal G}({\cal X}^2_W)/{\cal N}({\cal X}^2_W)$
are equivalent.
\begin{description}
\item[(1)]
The set of two-input functions $\{g_j\}$ form a basis of the quotient space 
${\cal G}({\cal X}^2_W)/{\cal N}({\cal X}^2_W)$.
\item[(2)]
The set $\{W_{\vec{\theta}}\}_{\vec{\theta} \in \bR^l}$
equals the set of transition matrices with the support 
${\cal X}^{k}_W$, i.e.,
$
 \{W_{\vec{\theta}}\}_{\vec{\theta} \in \bR^l}
=
{\cal W}_{\cX,W}$.
\item[(3)]
We have 
$\{ P_{W_{\vec{\theta}}}^2 \}_{\vec{\theta} \in \bR^l} 
=
\left\{ P\left| 
\begin{array}{l}
P \hbox{ is a distribution with support } {\cal X}^2_W \\
\hbox{to satisfy the stationary condition \eqref{symB}.}
\end{array}
\right.\right\}$.
\end{description}
\end{lemma}

\begin{proof}
\PF{${\bf (1)} \Leftrightarrow {\bf (2)}$}
Any element $W'\in {\cal W}_{\cX,W}$ can be written as
$W'(x|x')= W(x|x')e^{g(x,x')}$ by using an element 
$g \in {\cal G}({\cal X}^2_W)$
because of the relation $\log \frac{W'(x|x')}{W(x|x')}\in {\cal G}({\cal X}^2_W)$.
Since the relation
$ \{W_{\vec{\theta}}\}_{\theta \in \bR^l}
\subset {\cal W}_{\cX,W}$ holds,
we have the equivalence relation between conditions {\bf (1)} and {\bf (2)}.

\PF{${\bf (2)} \Leftrightarrow {\bf (3)}$}
The equivalence between conditions {\bf (2)} and {\bf (3)} follows from Lemma \ref{STR3W}.
\end{proof}

This lemma shows that 
${\cal W}_{\cX,W}$ is an exponential family.

\begin{lemma}\Label{STR2W}
Given generators $g_1, \ldots, g_l$, 
we define the linear map $F^2$ from the set of first derivatives of 
$P^2_{W_{\vec{\theta}}}$ to $\mathbb{R}^{l}$ as
\begin{align}
F^2\Bigl( 
\frac{d}{dt}P^2_{W_{\vec{\theta} + \vec{c}(t)}} \Big|_{t=0}\Bigl) 
:=
\Big(\sum_{(x,x')\in {\cal X}^2_W}
g_j(x,x')\sum_{i=1}^l 
\frac{d c^i(t)}{dt}
\frac{\partial}{\partial \theta^i}
P^2_{W_{\vec{\theta}}} (x,x')\Big|_{t=0}
 \Big)_{j=1, \ldots,l},
\end{align}
where $\vec{c}$ is a differentiable function from $\bR$ to $\bR^l$ such that $\vec{c}(0)=0$. 
Then, the linear map $F^2$ is invertible.
\end{lemma}

\begin{proof}
The expectation parameter $\eta_j(\vec{\theta})$ satisfies 
\begin{align}
\frac{\partial^2 \phi(\vec{\theta})}{\partial \theta^i \partial \theta^j }
=
\frac{\partial}{\partial \theta^i}
\eta_j(\vec{\theta})
=\sum_{(x,x')\in {\cal X}^2_W}
g_j(x,x')
\frac{\partial}{\partial \theta^i}
P^2_{W_{\vec{\theta}}}(x,x').
\end{align}
Since the Hessian $\frac{\partial^2 \phi(\vec{\theta})}{\partial \theta^i \partial \theta^j }
$ is a positive definite matrix, we have the desired statement.
\end{proof}

To understand the quotient space ${\cal G}({\cal X}^2_W)/{\cal N}({\cal X}^2_W)$,
we define the map $W_*$ on ${\cal G}({\cal X}^2_W) $ as
$(g(x,x'))_{x,x' \in {\cal X}} \mapsto (g(x,x')W(x|x'))_{x,x' \in {\cal X}}$.
We denote the image of $W_*$ by ${\cal L}({\cal X}^2_W)$ 
although the image is the same as ${\cal G}({\cal X}^2_W)$.
Then, 
we introduce the subspaces ${\cal L}_1({\cal X}^2_W)$
and ${\cal G}_{1,W}({\cal X}^2_W)$ as
\begin{align*}
{\cal L}_1({\cal X}^2_W) := \{
B\in {\cal G}({\cal X}^2_W)
| B^T u_{{\cal X}}=0\} ,\quad
{\cal G}_{1,W}({\cal X}^2_W) := W_*^{-1}
{\cal L}_1({\cal X}^2_W) .
\end{align*}

\begin{lemma}\Label{L9-29}
Assume $W$ is irreducible.
For any element $[g']$ of
${\cal G}({\cal X}^{2}_W)/{\cal N}({\cal X}^{2}_W)$,
there uniquely exists an element $g $ of ${\cal G}_{1,W}( {\cal X}^{2}_W)$
such that $[g']=[g]$. 
Therefore, we can regard 
the space ${\cal G}_{1,W}( {\cal X}^{2}_W)$
as the quotient space
${\cal G}({\cal X}^{2}_W)/{\cal N}({\cal X}^{2}_W)$.
\end{lemma}

\begin{remark}
Lemma \ref{L9-29} was essentially proven in a more general form as \cite[Lemma 8]{LEH}.
Since its proof is composed of more complicated notations, 
we give a proof of Lemma \ref{L9-29} with current notations for readers' convenience
in Appendix \ref{AS2}.
Only Lemmas \ref{STR3W}, \ref{STRW}, and \ref{STR2W} are novel in this section.
Other statements in Section \ref{s3} are known.
\end{remark}

In particular, when $W$ is a positive transition matrix, i.e., ${\cal X}^{2}_{W}={\cal X}^{2}$,
the subspace ${\cal N}( {\cal X}^{2}_{W})$ does not depend on $W$
and is abbreviated to ${\cal N}({\cal X}^2)$.
In this case,
${\cal W}_{\cX,W}$ is the set of positive transition matrices.
Then, it does not depend on $W$, and is abbreviated to ${\cal W}_{\cX}$.

We define the Fisher information matrix for the natural parameter by 
the Hesse matrix
$\mathsf{H}_{\vec{\theta}} [\phi]
=[\frac{\partial^2 \phi}{\partial \theta^i\partial \theta^j}(\vec{\theta})]_{i,j}$.
The Fisher information matrix for the expectation parameter 
$\vec{\eta}(\vec{\theta})$
is given as 
$\mathsf{H}_{\vec{\theta}} [\phi]^{-1}$.
It is known that the inverse of Fisher information matrix gives the minimum mean square error of estimation under the respective parameters \cite[Section 8]{HW-est}.
Further, 
for fixed values $\theta_o^{k+1},\ldots,\theta_o^l$,
we call the subset $\{W_{\vec{\theta}}\in {\cal E}|
\vec{\theta}=(\theta^1,\ldots,\theta^{k},\theta_o^{k+1},\ldots,\theta_o^l)
\}$ an exponential subfamily of ${\cal E}$.

In the above discussion, we denote an element of the tangent space by an element $g$ or $[g]$ of 
${\cal G}_{1,W}({\cal X}_W^2) $ or ${\cal G}({\cal X}_W^2) /{\cal N}({\cal X}_W^2) $.
However, it is possible to express an element of the tangent space by an element $W_* g$
of $ {\cal L}_{1}({\cal X}_W^2)$, i.e., 
$(W_* g)(x,x')=g(x,x')W(x|x')$.
The former expression is called the exponential representation ($e$-representation),
and the latter is called the mixture representation ($m$-representation).

\subsection{Mixture family}\Label{s3-3}
In the following, we characterize the set of transition matrices
when we have a part of information.
For this aim, we assume that
the functions $\{g_j\}$ satisfy the condition of Proposition \ref{L1-14-2}.
For fixed values $\eta_{o,1},\ldots,\eta_{o,k}$,
we call the subset $\{W_{\vec{\theta}}\in {\cal E}|
\vec{\eta}(\vec{\theta})=(\eta_{o,1},\ldots,\eta_{o,k},\eta_{k+1},\ldots,\eta_l)
\}$ a {\it mixture subfamily of ${\cal E}$}.
Given a transition matrix $W$, 
when the expectation values of real-valued functions $g_{j}$ on $\cX^2$ are known to be real numbers $b_j$,
we consider that the true transition belongs to 
the set $\{V \in {\cal W}_{\cX,W}| \sum_{x,x'}g_j(x,x')V(x|x')P_{V}(x')=b_j \forall j \}$, which is called 
a {\it mixture family on ${\cal X}^2_W$ generated by the constraints $\{g_j=b_j\}$}.
Note that a mixture family on ${\cal X}^2_W$ 
does not necessarily contain $W$
because its definition depends on the real numbers $b_j$.
When $W$ is a positive transition matrix, 
it is simply called a {\it mixture family generated by the constraints 
$\{g_j=b_j\}$}
because ${\cal W}_{\cX,W}$ is the set of positive transition matrices.
For a given transition matrix $W$ and 
two mixture families ${\cal M}_1$ and ${\cal M}_2$ on ${\cal X}^2_W$,
the intersection ${\cal M}_1 \cap {\cal M}_2$ is also a mixture family
on ${\cal X}^2_W$.

\subsection{Divergence}\Label{s3-4}
To discuss the difference between transition matrices, we characterize the divergence
by using the potential function $\phi(\vec{\theta})$
as follows. 
\begin{proposition}[\protect{\cite[Lemma 4.3]{HW-est}}]\Label{L7}
Two transition matrices 
${W}_{\vec{\theta}}$ and ${W}_{\vec{\theta}'}$ satisfy
\begin{align}
D({W}_{\vec{\theta}} \| {W}_{\vec{\theta}'})= &
\sum_{j=1}^l
(\theta^j-{\theta'}^j)\frac{\partial \phi}{\partial \theta^j}(\vec{\theta})- \phi(\vec{\theta})+ \phi(\vec{\theta}') \Label{1-1}.
\end{align}
\end{proposition}

The Fisher information matrix 
$\mathsf{H}_{\vec{\theta}} [\phi]$
can be characterized by the limit of the 
divergence as follows.

\begin{proposition}[\protect{\cite[Lemma 4.4]{HW-est}}] \Label{L20}
For $\vec{c}=(c^1, \ldots, c^l)$, 
we have
\begin{align}
\Label{27-20}
\lim_{t \to 0}
\frac{2}{t^2}D({W}_{\vec{\theta}} \| {W}_{\vec{\theta}+\vec{c}t})
=&
\lim_{t \to 0}
\frac{2}{t^2}D({W}_{\vec{\theta}+\vec{c}t} \| {W}_{\vec{\theta}})
=\sum_{i,j}\mathsf{H}_{\vec{\theta}} [\phi]_{i,j}c^i c^j .
\end{align}
\end{proposition}

The right hand side of (\ref{1-1}) 
can be regarded as the Bregman divergence \cite{Br}\footnote{Amari-Nagaoka \cite{AN} also defined the same quantity as 
the Bregman divergence with the name ``canonical divergence.''}
of the strictly convex function $\phi(\vec{\theta})$.
In the following, we derive several properties of the divergence
by using Bregman divergence.
That is, the following properties follow only from the strong 
convexity of $\phi(\vec{\theta})$ and 
the properties of Bregman divergence.

Using \cite[(40)]{Am}, we have
another expression of $ D({W}_{\vec{\theta}} \| {W}_{\vec{\theta}'})$
as 
\begin{align}
D(W_{\vec{\theta}(\vec{\eta})} \| W_{\vec{\theta}(\vec{\eta}')})
=\sum_{j=1}^l \theta(\vec{\eta}')^j(\eta_j'-\eta_j) 
- \nu(\vec{\eta}') +\nu(\vec{\eta}),
\Label{5-28-1}
\end{align}
where $\nu(\vec{\eta})$ is defined as Legendre transform of 
$\phi(\vec{\theta}) $ as
\begin{align}
\nu(\vec{\eta})
&:= \max_{\vec{\theta}} \sum_{i=1}^l  \theta^i \eta_i - \phi(\vec{\theta}) 
= \sum_{i=1}^l \theta^i(\vec{\eta}) \eta_i - \phi(\vec{\theta}(\vec{\eta}))\Label{eq1-28-1} \\
\vec{\theta}(\vec{\eta}) &=
\argmax_{\vec{\theta}} \sum_{i=1}^l  \theta^i \eta_i - \phi(\vec{\theta}) .\Label{eq1-28-2}
\end{align}
Since $\nu(\vec{\eta})$ is convex as well as $\phi(\vec{\theta})$,
we have the following lemma.
\begin{proposition}[\protect{\cite[Lemma 4.5]{HW-est}}] \Label{L1-14}
(1) For a fixed $\vec{\theta}$, 
the map $\vec{\theta}' \mapsto
D(W_{\vec{\theta}} \| W_{\vec{\theta}'})$
is convex.
(2) For a fixed $\vec{\theta}'$, 
the map $\vec{\eta} \mapsto
D(W_{\vec{\theta}(\vec{\eta})} \| W_{\vec{\theta}'})$ is convex.
\end{proposition}

It is known that Bregman divergence satisfies the Pythagorean theorem \cite{AN}\cite[(34)]{Am}\footnote{The derivation 
of the Pythagorean theorem for Bregman divergence 
is also available in \cite[Section 2.2.2]{H2nd}.}.
Here, we should remark that Bregman divergence is defined for a strictly convex function, not for a distribution family.
This kind of generality of Bregman divergence is the key point for information geometry discussed in this paper.
As mentioned in \cite[Proposition 4.6]{HW-est},
applying this fact, we have the following proposition as the Pythagorean theorem.
\begin{proposition}[\protect{\cite[(23)]{HN}}]\Label{P1-15}
We focus on two points 
$\vec{\theta}'=({\theta'}^1,\ldots,{\theta'}^l)$
and $\vec{\theta}=({\theta}^1,\ldots,{\theta}^l)$.
We choose 
the exponential subfamily of ${\cal E}$ whose natural parameters
$\theta^{k+1},\ldots,\theta^l $ are fixed to 
${\theta}^{k+1},\ldots,{\theta}^l$,
and
the mixture subfamily of ${\cal E}$ whose expectation parameters
$\eta_1,\ldots,\eta_k $ are fixed to 
$\eta_1(\vec{\theta}'),\ldots,\eta_k(\vec{\theta}') $.
Let ${\vec{\theta}''}=({\theta''}^1,\ldots,{\theta''}^l)$
be the natural parameter of the intersection of these two subfamilies of ${\cal E}$.
That is, 
${\theta''}^{j}={\theta}^{j}$ for $j=k+1,\ldots,d$
and
$\eta_j({\vec{\theta}''})=\eta_j(\vec{\theta}') $ for $k=1, \ldots,k$.
Then, we have 
\begin{align}
\Label{5-1}
D({W}_{\vec{\theta}'} \| {W}_{\vec{\theta}})
=
D({W}_{\vec{\theta}'} \| {W}_{{\vec{\theta}''}})+D({W}_{{\vec{\theta}''}} \| {W}_{\vec{\theta}}).
\end{align}
\end{proposition}

\subsection{Central limit theorem}\Label{s3-6}
Given an irreducible transition matrix $W$ and a general two-input function $g\in {\cal G}({\cal X}^2_W)$,
we consider 
the random variable $g^n(X^{n+1}) :=\sum_{k=1}^{n}g(X_{k+1},X_k)$ 
when the random variables $X^{n+1}:=(X_{n+1},\ldots, X_1)$ 
are subject to the joint distribution
\begin{align}
(W^{\times n} \times P )(x_n,\ldots, x_1)
:= W(x_{n+1}|x_{n})\cdots W(x_2|x_{1}) P(x_1)
\end{align}
with an arbitrary initial distribution $P$ on ${\cal X}$.

To discuss the expectation and the variance, 
we introduce the notations
$\mathsf{E}_{P,W}$ and $\mathsf{V}_{P,W}$, which describe
the expectation and the variance under the distribution $W^{\times n} \times P$,
respectively.
We also denote the logarithm of the Perron Frobenius eigenvalue of the matrix
$e^{s g(x|x')}W(x|x')$
by $\varphi(s)$.
The following proposition is known as the central limit theorem\footnote{Its detailed review is available in \cite[Remark 7.3]{HW14-2}}.

\begin{proposition}[\cite{Ben-Ari,CLT2,CLT4,Lalley,CLT3,HW14-2}]
The relations 
\begin{align}
\lim_{n \to \infty}
\frac{1}{n}\mathsf{E}_{P,W} [g^n(X^{n+1})] &= \langle u_{{\cal X}}| (W_* g)P_W\rangle 
=\frac{d}{d s}\varphi(s)|_{s=0},
\\
\lim_{n \to \infty}
\frac{1}{n}\mathsf{V}_{P,W} [g^n(X^{n+1})] &= 
\frac{d^2}{d s^2}\varphi(s)|_{s=0}
\end{align}
hold. Further, the random variable
$\frac{1}{\sqrt{n}}( g^n(X^{n+1})- n \frac{d}{d s}\varphi(s)|_{s=0})$
asymptotically obeys 
the zero-mean Gaussian distribution with variance $ \frac{d^2}{d s^2}\varphi(s)|_{s=0}$.
\end{proposition}

Now, we consider a more general case, in which, multiple functions 
$\tilde{g}_1, \ldots, \tilde{g}_k \in {\cal G}({\cal X}^2_W)$.
Given $\vec{s}:=(s^1, \ldots, s^k)$, 
we denote the logarithm of the Perron Frobenius eigenvalue of the matrix
$e^{\sum_{i=1}^k s^i \tilde{g}_i(x|x')}W(x|x')$
by $\phi(\vec{s})$.
So, the above lemma can be generalized as follows.

\begin{proposition}[\protect{\cite[Theorem 8.2]{HW-est}}]\Label{P2-9-1}
The random variables
$\frac{1}{\sqrt{n}}( \tilde{g}_j^n(X^{n+1})- 
n \frac{\partial}{\partial s^j} \phi(\vec{s})|_{\theta=0})$
asymptotically obey the Gaussian distribution with average $0$.
The variance is given by the Hessian of $\phi$ at $s=0$, i.e.,
$ \frac{\partial^2}{\partial s^i \partial s^j}\phi(s)|_{s=0}$.
\end{proposition}

Now, we consider the exponential family with generator $g_1, \ldots, g_k \in {\cal G}({\cal X}_W)$.
Since the expectation of 
 $\frac{1}{n}\mathsf{E}_{P,W_{\vec{\theta}}} [g_i^n(X^{n+1})] $
converges to the expectation parameter 
$\frac{\partial }{\partial \theta^i}\phi(\vec{\theta})
|_{\vec{\theta}=0} =\eta_i(\vec{\theta})$,
the random variable 
$\frac{g_i^n(X^{n+1})}{n}$ works as an estimator of the expectation parameter.
Its asymptotic variance is the Hessian of $\phi$,
which is the minimum variance under suitable conditions for estimator \cite[Section 8.2]{HW-est}.

\section{Partial observation model}\Label{s4}
In this section, we consider how to estimate the parameter 
describing the transition matrix when our observation is restricted to some of generators.
As seen in Section \ref{s6B1}, 
this model can be regarded as a generalization of estimation of hidden Markov process.
That is, we extend the contents of Subsection \ref{s3-6} to such a restricted case.
For this purpose, we introduce additional notations.
Given positive integers $l_1,l_2,l_3$, 
we denote the vector spaces ${\cal V}_j:=\mathbb{R}^{l_j}$ for $j=1,2,3$,
and the vector spaces 
${\cal V}_{1,2,3}:={\cal V}_1\oplus {\cal V}_2\oplus {\cal V}_3$, and 
${\cal V}_{1,2}:={\cal V}_1\oplus {\cal V}_2$, and 
${\cal V}_{2,3}:={\cal V}_2\oplus {\cal V}_3$.
In the following, we regard the spaces ${\cal V}_j$ and ${\cal V}_{i,j}$
as subspaces of ${\cal V}_{1,2,3}$, respectively. 
We denote the coordinate projection from ${\cal V}_{1,2,3}$ to the subspaces ${\cal V}_j$ and ${\cal V}_{i,j}$ by
the $l_j\times (l_1+l_2+l_3)$ matrix $P_j$ and 
the $(l_i+l_j)\times (l_1+l_2+l_3)$ matrix $P_{i,j}$, respectively. 
For a vector $\vec{v} \in {\cal V}_{1,2,3}$, we denote $P_j \vec{v}$ and $P_{i,j} \vec{v}$
by $\vec{v}_j$ and $\vec{v}_{i,j}$, respectively.

Let $W$ be an irreducible transition matrix on ${\cal X}$,
and $g_1, \ldots, g_{l_1+l_2+l_3}$ be linearly independent functions as elements of
${\cal G}({\cal X}^2_{W})/{\cal N}({\cal X}^2_{W})$.
So, we define the potential function $\phi$ by using the generator $\{g_1, \ldots, g_{l_1+l_2+l_3}\}$
in the way given in Section \ref{s3-2}.
Let ${\cal W}:=\{W_{\vec{\theta}_1,\vec{\theta}_{2,3}}\}$ be
the exponential family of transition matrices on ${\cal X}$ by $W$ with the above generator.
Then, we define the expectation parameter 
$\eta_j(\vec{\theta}_1,\vec{\theta}_{2,3}):=
\frac{\partial \phi}{\partial \theta^j}(\vec{\theta}_1,\vec{\theta}_{2,3})$.
Now, we consider the exponential subfamily 
${\cal W}_2:=\{W_{0,\vec{\theta}_{2,3}}\}_{\Theta_{2,3} \in \Theta_{2,3}}$,
where the parametric space $\Theta_{2,3}\subset {\cal V}_{2,3}$ will be discussed in a careful way.

Now, we assume that we can observe only the sample mean 
of the part of generators $g_1, \ldots, g_{l_1+l_2}$ with $n+1$ observations.
That is, we can observe the sample mean 
$\vec{Y}_{1,2}^n:=\big(\frac{{g}_1(X^{n+1})}{n}, \ldots,\frac{{g}_{l_1+l_2}(X^{n+1})}{n}\big)$,
where ${g}_j(X^{n+1}):=\sum_{i=1}^n {g}_j(X_{i+1}, X_{i})$.
So, Proposition \ref{P2-9-1} guarantees that
the expectation of the sample mean $\vec{Y}_{1,2}^n$ 
converges to $\vec{\eta}_{1,2}(0,\vec{\theta}_{2,3})$ 
as $n$ goes to infinity with whatever initial distribution
when the true transition matrix is $W_{0,\vec{\theta}_{2,3}}$. 
Now, we propose an estimator 
$\hat{\theta}_{2,3}(\vec{Y}_{1,2}^n) $ for $\vec{\theta}_{2,3}$
from the observed sample mean $\vec{Y}_{1,2}^n$ as follows (Fig. \ref{F1}).
\begin{align}
\hat{\theta}_{2,3}(\vec{Y}_{1,2}^n) :=
\argmin_{\vec{\theta}_{2,3}' \in \Theta_{2,3}}
\min_{\vec{\eta}_3' \in {\cal V}_3}
D\big(W_{\vec{\theta}(\vec{Y}_{1,2}^n,\vec{\eta}_3')}\big\|W_{0,\vec{\theta}_{2,3}'}\big).
\Label{eq6-9-1}
\end{align}
For the definition of $\vec{\theta}(\vec{Y}_{1,2}^n,\vec{\eta}_3')$, see \eqref{eq1-28-2}.
Here, $\argmin_{\vec{\theta}_{2,3}' \in \Theta_{2,3}}$
is not unique in general.
In this case, we choose one of elements $\vec{\theta}_{2,3}' \in \Theta_{2,3}$ attaining the minimum.

\begin{figure}[htbp]
\begin{center}
\scalebox{1}{\includegraphics[scale=0.5]{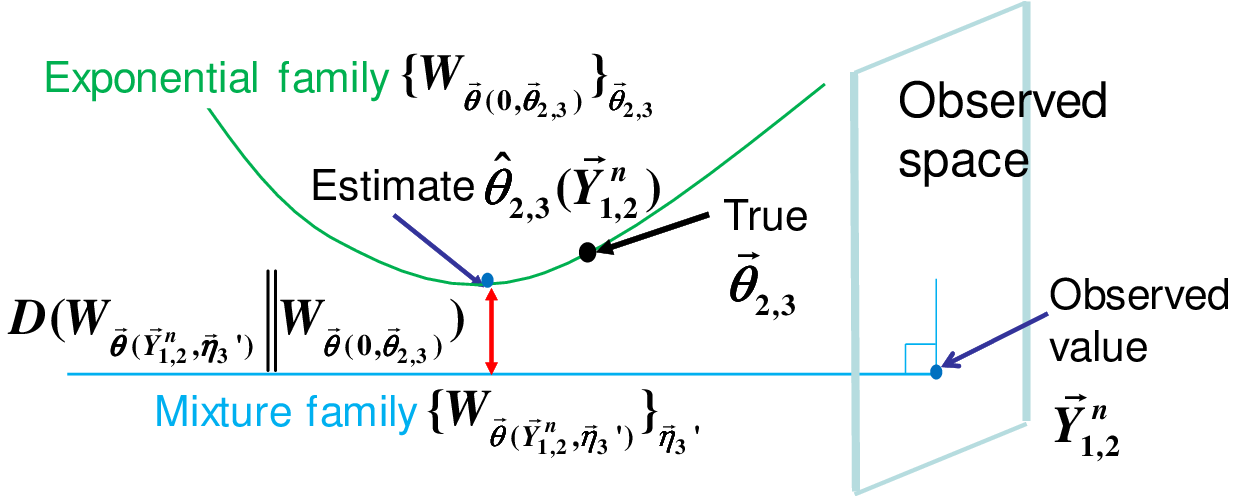}}
\end{center}
\caption{Estimator of partial observation model}
\Label{F1}
\end{figure}%

This estimator can be calculated approximately by using the em-algorithm in the following way.
Firstly, we fix the initial point 
$\vec{\theta}_{2,3;0} \in \Theta_{2,3}$.
Then, repeating the following procedure, we calculate 
$\vec{\theta}_{2,3;j} \in \Theta_{2,3}$ from 
$\vec{\theta}_{2,3;j-1} \in \Theta_{2,3}$ iteratively as follows.
\begin{description}
\item[E-step]
We find the e-minimum point $\vec{\eta}_{3;j}:=
\argmin_{\vec{\eta}_{3}}
D\big(W_{\vec{\theta}(\vec{Y}_1^n,\vec{\eta}_3)}\big\|W_{0,\vec{\theta}_{2,3;j-1}}\big)$.
\item[M-step]
Next, we find m-minimum point $\vec{\theta}_{2,3;j}:=
\argmin_{\vec{\theta}_{2,3}}
D\big(W_{\vec{\theta}(\vec{Y}_{1,2}^n,\vec{\eta}_{3;j})}\big\|W_{0,\vec{\theta}_{2,3}}\big)$.
\end{description}
The implementation is discussed in Appendix \ref{ASImp}.
Since the maps 
$\vec{\eta}_3 \mapsto
D\big(W_{\vec{\theta}(\vec{Y}_1^n,\vec{\eta}_{3})}\big\|W_{0,\vec{\theta}_{2,3;j-1}}\big)$
and
$\vec{\theta}_{2,3} \mapsto
D\big(W_{\vec{\theta}(\vec{Y}_{1,2}^n,\vec{\eta}_{3;j})}\big\|W_{0,\vec{\theta}_{2,3}}\big)$
are convex due to Proposition \ref{L1-14}, 
these can be calculated by the convex optimization.

Pythagorean theorem (Proposition \ref{P1-15}) guarantees the equations
\begin{align}
&D\big(W_{\vec{\theta}(\vec{Y}_{1,2}^n,\vec{\eta}_{3;j})}\big\|W_{0,\vec{\theta}_{2,3;j-1}}\big)
=
D\big(W_{\vec{\theta}(\vec{Y}_{1,2}^n,\vec{\eta}_{3;j-1})}\big\|W_{0,\vec{\theta}_{2,j-1}}\big)
-
D\big(W_{\vec{\theta}(\vec{Y}_{1,2}^n,\vec{\eta}_{3;j-1})}\big\|
W_{\vec{\theta}(\vec{Y}_{1,2}^n,\vec{\eta}_{3;j})}\big), \\
&D\big(W_{\vec{\theta}(\vec{Y}_{1,2}^n,\vec{\eta}_{3;j})}\big\|W_{0,\vec{\theta}_{2,3;j}}\big)
=
D\big(W_{\vec{\theta}(\vec{Y}_{1,2}^n,\vec{\eta}_{3;j})}\big\|W_{0,\vec{\theta}_{2,3;j-1}}\big)
-
D\big(W_{0,\vec{\theta}_{2,3;j}}\big\|W_{0,\vec{\theta}_{2,3;j-1}}\big).
\end{align}
So, repeating this procedure, we can achieve a local minimum value.

For a given point $\vec{\theta}_{2,3;o} \in \Theta_{2,3}$,  
we fix $\vec{\eta}_{1,o}:=\vec{\eta}_1(0,\vec{\theta}_{2,3;o})$.
We denote the Hessian for the potential $\phi$ with respect to the parameters
$(\vec{\theta}_1,\vec{\theta}_{2,3})$ at $(0,\vec{\theta}_{2,3;o})$
by $\mathsf{H}_{0,\vec{\theta}_{2,3;o}}[\phi] $.
We define the $(l_1+l_2)\times (l_2+l_3)$ Jacobi matrix $A(\vec{\theta}_{2,3;o})$ whose 
$(i,j)$ component is $\frac{\partial \eta_i(0,\vec{\theta}_{2,3;o})}{\partial \theta^{j}}$
($i=1, \ldots, l_1+l_2$ and $j=l_1+1, \ldots, l_1+l_2+l_3$).
Then, the Jacobi matrix $A(\vec{\theta}_{2,3;o})$ is calculated to be 
$ P_{1,2} \mathsf{H}_{0,\vec{\theta}_{2,3;o}}[\phi] P_{2,3}^T$.

Then, we define the projective Fisher information matrix 
$\tilde{\mathsf{H}}_{\vec{\theta}_{2,3;o}} $
at $\vec{\theta}_{2,3;o} $ by 
\begin{align}
\tilde{\mathsf{H}}_{\vec{\theta}_{2,3;o}} :=&
A(\vec{\theta}_{2,3;o})^T 
\Big(
 \mathsf{H}_{0,\vec{\theta}_{2,3;o}}[\phi] ^{-1} \nonumber \\
&-
\mathsf{H}_{0,\vec{\theta}_{2,3;o}}[\phi]^{-1} P_3^T
(P_{3} \mathsf{H}_{0,\vec{\theta}_{2,3;o}}[\phi]^{-1} P_{3}^T)^{-1}
P_3 \mathsf{H}_{0,\vec{\theta}_{2,3;o}}[\phi]^{-1}
\Big)
A(\vec{\theta}_{2,3;o}).\Label{FTE}
\end{align}

Then, we have the following lemma, whose proof is given in Appendix \ref{s4-3}.
\begin{lemma}\Label{6-13-thL}
The following conditions are equivalent.
\begin{description}
\item[B1]
The rank of $A(\vec{\theta}_{2,3;o})$ is $l_2+l_3$.
\item[B2]
The matrix $\tilde{\mathsf{H}}_{\vec{\theta}_{2,3;o}} $
is invertible.
\end{description}
\end{lemma}

We call the condition of Lemma \ref{6-13-thL} the {\it Jacobi matrix condition}
because the matrix $A(\vec{\theta}_{2,3;o})$ is the Jacobi matrix.
Now, we impose the following condition to the parametric space $ \Theta_{2,3}$;
\begin{description}
\item[C1]
The parametric space $\Theta_{2,3}$ is an open connected subset of $\bR^{l_2+l_3}$.
\item[C2]
The map $\vec{\theta}_{2,3} \in \Theta_{2,3} \mapsto \vec{\eta}_{1,2}(0,\vec{\theta}_{2,3})$ is one-to-one.
\item[C3]
The matrix $A(\vec{\theta}_{2,3})$ has rank $l_2+l_3$ for any element $\vec{\theta}_{2,3} \in \Theta_{2,3} $.
That is, the projective Fisher information matrix 
$\tilde{\mathsf{H}}_{\vec{\theta}_{2,3;o}} $
is invertible.
\end{description}
To satisfy these conditions, the number $l_3$ needs to be smaller than $l_1$.
The feasibility of the above assumption will be discussed in Theorem \ref{T3-29}.

Then, to refine Theorem \ref{PT1}, as an extension of Proposition \ref{P2-9-1}, we obtain the following theorem, whose proof is given in Appendix \ref{s4-3}.
\begin{theorem}\Label{6-13-th}
Assume that the exponential family $ \{ W_{0,\vec{\theta}_{2,3}}\}_{\vec{\theta}_{2,3} \in \Theta_{2,3}}$
satisfies the conditions {\bf C1}, {\bf C2}, and {\bf C3}.
\mh{If} the true natural parameter is $\vec{\theta}_{2,3;o}$,
then the random variable 
$\sqrt{n}(\hat{\theta}_{2,3}(\vec{Y}_{1,2}^n) - \vec{\theta}_{2,3;o})$
asymptotically obeys
the zero-mean Gaussian distribution of covariance matrix 
$
\tilde{\mathsf{H}}_{\vec{\theta}_{2,3;o}}^{-1}$.
\end{theorem}

To discuss the validity of our model or our estimate, 
as a precise statement of Theorem \ref{T26},
we prepare the following theorem, whose proof is given in Appendix \ref{s4-3}.

\begin{theorem}\Label{6-13-thB}
We assume the following conditions
\begin{description}
\item[D1] $\Theta_{2,3}$ is an open set.
\item[D2] The neighborhood of any element of  $\Theta_{2,3}$ satisfies the same condition as Theorem \ref{6-13-th},
\item[D3]
The integer $\rank A(\vec{\theta}_{2,3})$ is a constant value for 
$\vec{\theta}_{2,3}\in \Theta_{2,3}$. 
So, we define the integer $l[\Theta_{2,3}]:=\rank A(\vec{\theta}_{2,3})-l_2$.
 \item[D4] The true parameter belongs to $\Theta_{2,3}$.
 \end{description}
Then, the random variable 
$2n \min_{\vec{\theta}_{2,3}' \in \Theta_{2,3}}
\min_{\vec{\eta}_3' \in {\cal V}_3}
D(W_{\vec{\theta}(\vec{Y}_{1,2}^n,\vec{\eta}_3')}\|W_{0,\vec{\theta}_{2,3}'})$
is subject to $\chi^2$-distribution with degree 
$l_1-l[\Theta_{2,3}]$.
\end{theorem}

Under the assumption of Theorem \ref{PT1},
if the true transition matrix contains the model ${\cal W}$
and the number $n$ of observations is sufficiently large,
the value $2n \min_{\vec{\eta}_3 \in {\cal V}_3}
D(W_{\vec{\theta}(\vec{Y}_{1,2}^n,\vec{\eta}_3)}\|W_{0,\hat{\theta}_{2,3}(\vec{Y}_{1,2}^n)})$
is approximately subject to $\chi^2$-distribution with degree 
$l_1-l[\Theta_{2,3}]$.
If this value obtained by the above em-algorithm 
is very far from typical values in $\chi^2$-distribution with degree 
$l_1-l[\Theta_{2,3}]$,
we need to perform the above em-algorithm again with another initial value.
If the value still does not the above condition after several trials of the em-algorithm,
we need to consider the possibility that our model ${\cal W}$ is not correct.

Hence, when the number of observation $n$ is sufficiently large,
the random variable given in Theorem \ref{T26} can be used for the $\chi^2$-test.
Thus, we can choose a suitable model among several models when we fixed a certain significance level $\alpha>0$.
We choose the model with the minimum dimension 
among models passing this $\chi^2$-test with the significance level $\alpha$.
Using Theorem \ref{6-13-thB}, 
we can apply $\chi^2$-test to the hypothesis testing with
\begin{description}
\item[H0: Null hypothesis]
The true transition matrix is contained in our model ${\cal W}$.
\item[H1: Alternative hypothesis]
The true transition matrix is not contained in our model ${\cal W}$.
\end{description}

When the risk probability is $\alpha$ and 
the observed value\par\noindent 
$\min_{\vec{\eta}_{3}' \in {\cal V}_3}
D(W_{\vec{\theta}(\vec{Y}_{1,2}^n,\vec{\eta}_3')}\|W_{0,\hat{\theta}_{2,3}(\vec{Y}_{1,2}^n)})$
is greater than $F_{l_1-l[\Theta_{2,3}]}^{-1}(1-\alpha)$,
we can reject the null hypothesis $H0$,
where $F_{l}$ is the cumulative distribution function of $\chi^2$-distribution with degree $l$.
In other words,
when we observe the value 
$\chi^2:=\min_{\vec{\eta}_3' \in {\cal V}_3}
D(W_{\vec{\theta}(\vec{Y}_{1,2}^n,\vec{\eta}_3')}\|W_{0,\hat{\theta}_{2,3}(\vec{Y}_{1,2}^n)})$,
we can reject the null hypothesis $H0$ with risk probability
$1-F_{l_1-l[\Theta_{2,3}]}(\chi^2)$.
In this way, we can estimate the validity of our model even though the true point is a singular point. 

\section{$k-1$-memory transition matrices}\Label{s4-5}
\subsection{$k-1$-memory transition matrix and joint distribution}\Label{s4-5-1}
To apply the partial observation model to the hidden Markovian model,
we consider the case when the distribution of the outcome on ${\cal X}$ depends on the previous $k-1$ outcomes.
Such a sequence is called a Markovian chain with order $k-1$, and is given as a $k-1$-memory transition matrix on ${\cal X}$,
which is described as
$W=\{W(x_{k}|x_{k-1}, \ldots, x_1)\}_{x_{k}\in {\cal X}, (x_{1}, \ldots, x_{k-1})\in {\cal X}^{k-1}}$. 
That is, the $k$-th distribution depends on 
the previous $k-1$ outputs $(x_{k-1}, \ldots, x_1)$.
Then, we denote the support of $W$ by
${\cal X}^{k}_W:=
\{(x_{k},x_{k-1}, \ldots, x_1) \in {\cal X}^{k} |
W(x_{k}|x_{k-1}, \ldots, x_1) >0\}$.
In this definition, a $k-1$-memory transition matrix can be regarded as a $k$-memory transition matrix.


The $k-1$-memory transition matrix $W$ can be naturally regarded as
a transition matrix $W_{|{\cal X}^{k-1}}$ on ${\cal X}^{k-1}$ as
$$W_{|{\cal X}^{k-1}}(x_{k-1}, \ldots, x_1|x_{k-1}', \ldots, x_1'):=
W(x_{k-1}|x_{k-1}', \ldots, x_1')\delta_{x_{k-2},x_{k-1}'}\cdots \delta_{x_{1},x_{2}'}.$$
We say that a $k-1$-memory transition matrix $W$ is irreducible (ergodic)
when the transition matrix $W_{|{\cal X}^{k-1}}$ is an irreducible (ergodic) matrix.
Then, we also define the divergence of two
$k-1$-memory transition matrices $W$ and $W'$
as $D(W\|W'):=D(W_{|{\cal X}^{k-1}}\|W_{|{\cal X}^{k-1}}')$.
Thus, using the stationary distribution 
$P_{W}$ on ${\cal X}^{k-1}$, we have
\begin{align}
D(W\|W')= 
\sum_{x_{k-1}, \ldots, x_1}P_{W}(x_{k-1}, \ldots, x_1)
D(W_{x_{k-1}, \ldots, x_1}
\|W_{x_{k-1}, \ldots, x_1}').
\end{align}

Now, we introduce the stationary condition for a distribution $P$ on ${\cal X}^k$ as
\begin{align}
\sum_{x' \in {\cal X}} P(x', x_{k-1},\ldots, x_1)
=\sum_{x' \in {\cal X}} P(x_{k-1},\ldots, x_1,x').\Label{sym}
\end{align}
Also, we introduce 
the joint distribution $P_W^k$
and
the set of transition matrices ${\cal W}_{\cX^{k-1},W}$
as
\begin{align*}
&P_W^k(x_k, x_{k-1},\ldots, x_1) := W(x_k| x_{k-1},\ldots, x_1) P_W( x_{k-1},\ldots, x_1) \\
&{\cal W}_{\cX^{k-1},W} := \{V| V \hbox{ is a $k-1$-memory transition matrix and }
{\cal X}^k_{W}={\cal X}^k_V \}.
\end{align*}
Since $ ({\cal X}^{k-1})^2_{W_{|{\cal X}^{k-1}}}$ has one-to-one relation with
${\cal X}^{k}_W$ as
\begin{align}
&({\cal X}^{k-1})^2_{W_{|{\cal X}^{k-1}}}
=
\left\{
(x_{k-1}, \ldots, x_1,x_{k-1}', \ldots, x_1') \in {\cal X}^{2k-2}
\left|
\begin{array}{l}
(x_{k-1}, x_{k-1}', \ldots, x_1') \in {\cal X}^{k}_W\\
x_{k-2}=x_{k-1}', \ldots, x_1=x_2'
\end{array}
\right.
\right\},\Label{SGT}
\end{align}
applying Lemma \ref{STR3W} to the transition matrix $W_{|{\cal X}^{k-1}}$ on ${\cal X}^{k-1}$, 
we find the following lemma.

\begin{lemma}\Label{STR3}
The set of distributions $P_W^k$ equals that the set of 
joint distributions on ${\cal X}^k$ satisfying the stationary condition \eqref{sym}.
More precisely, 
\begin{align*}
\{ P_{W'}^k | W' \in {\cal W}_{\cX^{k-1},W}\} 
=\left\{ P\left| 
\begin{array}{l}
P \hbox{ is a distribution with support } {\cal X}^k_W \\
\hbox{to satisfy the stationary condition \eqref{sym}.}
\end{array}
\right.\right\}
\end{align*}
\end{lemma}

The next lemma shows that an ergodic $k-1$-memory transition matrix 
can be realized in a natural setting.

\begin{lemma}\Label{STR3-2}
Assume that a $k-1$-memory transition matrix $W$ is ergodic.
\mh{If} $W$ is regarded as a $k$-memory transition matrix,
then any initial distribution on ${\cal X}^k$
converges to the joint distribution $P_W^k$. 
Hence, $W$ is ergodic as a $k$-memory transition matrix.
\end{lemma}

This lemma guarantees that any ergodic transition matrix can be regarded as 
an ergodic $k-1$-memory transition matrix $W$, inductively.

\begin{proof}
Whatever initial distribution is,
the distribution for $(x_k, \ldots, x_2)$ converges to the stationary distribution $P_W$.
When the marginal distribution for $(x_k, \ldots, x_2)$ of the initial distribution on ${\cal X}^{k}$ is
$P_W $,
the output distribution is $P_W^k$. 
Hence, we obtain the desired statement.
\end{proof}

\begin{remark}\Label{RT5}
Here, we should remark that 
an irreducible $k-1$-memory transition matrix $W$ is not necessarily 
irreducible as a $k$-memory transition matrix.
For simplicity, we consider the case with $k=2$.
Assume that ${\cal X}$ is $\mathbb{Z}_d$ and $W(x|x')=\delta_{x,x'+1}$, i.e.,
$W$ has a cyclic form.
Then, $W$ is irreducible, but is not ergodic.
When $W$ is regarded as a $2$-memory transition matrix, we have
$W_{{\cal X}^2}(x_2,x_1|x_2',x_1')= \delta_{x_2,x_2'+1}\delta_{x_1,x_2'}$.
Then, the subset $\{(2,1),(3,2), \ldots,(d,d-1),(1,d) \}$ is closed in the application of $W_{{\cal X}^2}$.
Hence, $W_{{\cal X}^2}$ is not irreducible.
\end{remark}

\subsection{Exponential family of $k-1$-memory transition matrices}\Label{s4-5-2}
Now, we define an exponential family of $k-1$-memory transition matrices
as a natural extension of an exponential family of transition matrices as follows.
Firstly, we fix an irreducible $k-1$-memory transition matrix $W$ on ${\cal X}$.
Then, we denote the linear space of real-valued functions 
$\{g(x_{k},x_{k-1}, \ldots, x_1)\}$ defined on ${\cal X}^{k}_W$
by ${\cal G}({\cal X}^{k}_W)$.
Additionally, ${\cal N}({\cal X}^{k}_W) $ expresses the subspace of functions with form
$ f(x_{k},x_{k-1}, \ldots,x_2)-f(x_{k-1},x_{k-2}, \ldots,x_1)+c$ for
$(x_{k-1}, x_{k-2}, \ldots, x_1 ) \in {\cal X}^k_{W}$.
When functions $g_1, \ldots, g_l \in {\cal G}({\cal X}^{k}_W)$
are linearly independent as elements of ${\cal G}({\cal X}^{k}_W)/{\cal N}({\cal X}^{k}_W)$,
for $\vec{\theta}:=(\theta^1, \ldots, \theta^l) \in \bR^l$,
we define the matrix 
\begin{align}
&\overline{W}_{\vec{\theta}|{\cal X}^{k-1}}(x_{k-1}, \ldots, x_1|x_{k-1}', \ldots, x_1')\nonumber \\
:=&
e^{\sum_{j=1}^l \theta^j g_j(x_{k-1},x_{k-1}', \ldots, x_1')}
W(x_{k-1}|x_{k-1}', \ldots, x_1')\delta_{x_{k-2},x_{k-1}'}\cdots \delta_{x_{1},x_{2}'},
\end{align}
and denote the Perron-Frobenius eigenvalue by $\lambda_{\vec{\theta}}$.
Also, we denote the Perron-Frobenius eigenvector of the transpose 
$\overline{W}_{\vec{\theta}|{\cal X}^{k-1}}^T$
by $\overline{P}^3_{\vec{\theta}}$.
Then, we define the $k-1$-memory transition matrix 
$W_{\vec{\theta}}$ as
\begin{align}
W_{\vec{\theta}}(x_{k}|x_{k-1}, \ldots, x_1) 
:=&
\lambda_{\vec{\theta}}^{-1}
\overline{P}^3_{\vec{\theta}}(x_{k},x_{k-1}, \ldots, x_2)
e^{\sum_{j=1}^l \theta^j g_j(x_{k},x_{k-1}, \ldots, x_1)}
W(x_{k}|x_{k-1}, \ldots, x_1)
\nonumber \\
& 
\overline{P}^3_{\vec{\theta}}(x_{k-1},x_{k-2}, \ldots, x_1)^{-1}.
\end{align}
That is, we call $\{W_{\vec{\theta}}\}_{\theta \in \Theta}$
the exponential family 
of $k-1$-memory transition matrices 
generated by 
the generators $g_1, \ldots, g_l $ at $W$.
The expectation parameter $\eta_j(\vec{\theta})$ with $j=1, \ldots, l$ is given as
\begin{align}
\eta_j(\vec{\theta})
:=\sum_{\vec{x}\in {\cal X}^k}
g_j(x_k,x_{k-1}, \ldots, x_1)
P^k_{W_{\vec{\theta}}}(x_k,x_{k-1}, \ldots, x_1 ).
\end{align}
Then, combining Lemma \ref{STR2W} and \eqref{SGT}, we have the following lemma.

\begin{lemma}\Label{STR2}
We define the linear map $F^k$ from the set of first derivatives of 
$P^k_{W_{\vec{\theta}}}$ to $\mathbb{R}^{l}$ as
\begin{align}
& F^k\Bigl( \sum_{i=1}^l a_i \frac{\partial }{\partial \theta^i}
P^k_{W_{\vec{\theta}}} \Bigl) 
:=
\Big(\sum_{\vec{x}\in {\cal X}^k}
g_j(x_k,x_{k-1}, \ldots, x_1)
\sum_{i=1}^l a_i\frac{\partial}{\partial \theta^i}
P^k_{W_{\vec{\theta}}}(x_k,x_{k-1}, \ldots, x_1 ) \Big)_{j=1, \ldots,l}.
\end{align}
Then, the linear map $F^k$ is invertible.
\end{lemma}

Also, the set $\{W_{\vec{\theta}|{\cal X}^{k-1}}\}_{\vec{\theta}}$
forms an exponential family of transition matrices on ${\cal X}^{k-1} $
because 
\begin{align}
&W_{\vec{\theta}|{\cal X}^{k-1}}(x_{k-1}, \ldots, x_1|x_{k-1}', \ldots, x_1')
\nonumber \\
=&
\lambda_{\vec{\theta}}^{-1}
\overline{P}^3_{\vec{\theta}}(x_{k-1}, \ldots, x_1)
\overline{W}_{\vec{\theta}|{\cal X}^{k-1}}
(x_{k-1}, \ldots, x_1|x_{k-1}', \ldots, x_1') 
\overline{P}^3_{\vec{\theta}}(x_{k-1}', \ldots, x_1')^{-1}.
\nonumber
\end{align}
In this sense, we call $\phi(\vec{\theta}):=\log \lambda_{\vec{\theta}}$
the potential function. 
 Then, similar to \eqref{1-1}, we have 
$
D({W}_{\vec{\theta}} \| {W}_{\vec{\theta}'})= 
\sum_{j=1}^d
(\theta^j-{\theta'}^j)\frac{\partial \phi}{\partial \theta^j}(\vec{\theta})- \phi(\vec{\theta})+ \phi(\vec{\theta}') \Label{1-1B}.
$

For linearly independent functions, 
the combination of Lemma \ref{STRW} and \eqref{SGT} yields the following lemma.

\begin{lemma}\Label{STR}
The following conditions for linearly independent function $g_1, \ldots, g_l$
in 
${\cal G}({\cal X}^k_W)/{\cal N}({\cal X}^k_W)$
are equivalent.
\begin{description}
\item[(1)]
The set of two-input functions $\{g_j\}$ form a basis of the quotient space 
${\cal G}({\cal X}^k_W)/{\cal N}({\cal X}^k_W)$.
\item[(2)]
The set $\{W_{\vec{\theta}}\}_{\vec{\theta} \in \bR^l}$
equals the set of transition matrices with the support 
${\cal X}^{k}_W$, i.e.,
 $\{W_{\vec{\theta}}\}_{\vec{\theta} \in \bR^l}
=
{\cal W}_{\cX^{k-1},W}.$\Label{SFH2W}
\item[(3)]
We have $
\{ P_{W_{\vec{\theta}}}^k \}_{\vec{\theta} \in \bR^l} 
=
\left\{ P\left| 
\begin{array}{l}
P \hbox{ is a distribution with support } {\cal X}^k_W \\
\hbox{to satisfy the stationary condition \eqref{sym}.}
\end{array}
\right.\right\}$.
\end{description}
\end{lemma}

This lemma shows that 
${\cal W}_{\cX^{k-1},W}$ is an exponential family.
When $W$ has no zero entry,
the dimension of ${\cal N}({\cal X}^{k}_W) $ is $|{\cal X}|^{k-1}$, hence,
the dimension of ${\cal G}({\cal X}^{k}_W)/{\cal N}({\cal X}^{k}_W) $ is $|{\cal X}|^{k}-|{\cal X}|^{k-1}$.

\subsection{Partial observation model of $k-1$-memory transition matrices}\Label{s5-2}
In this subsection, we rewrite the results in Section \ref{s4} 
for $k$-memory transition matrices,
which will work as preparation of our application of the partial observation model to the hidden Markovian model.
Let $W$ be an irreducible $k$-memory transition matrix on ${\cal X}$,
and
$g_1, \ldots, g_{l_1+l_2+l_3}$ be linearly independent functions as elements of
${\cal G}({\cal X}^{k}_{W})/{\cal N}({\cal X}^{k}_{W})$.
Let ${\cal W}:=\{W_{\vec{\theta}_1,\vec{\theta}_{2,3}}\}$ be
the exponential family of $k-1$-memory transition matrices on ${\cal X}$
by $W$ with the generator $\{g_1, \ldots, g_{l_1+l_2+l_3}\}$.
So, we can define the potential function $\phi$.
Now, we consider the exponential subfamily 
${\cal W}_2:=\{W_{0,\vec{\theta}_{2,3}}\}_{\Theta_{2,3} \in \Theta_{2,3}}$,
where the parametric space $\Theta_{2,3}\subset {\cal V}_{2,3}$ will be discussed in a careful way.

Now, we assume that we can observe only the sample mean 
of the a part of generators $g_1, \ldots, g_{l_1+l_2}$ with $n+k$ observations
$X^{n+k}= (X_1, \ldots, X_{n+k})$.
That is, we can observe the sample mean $\vec{Y}_{1,2}^n:=
(\frac{{g}_1(X^{n+k})}{n}, \ldots,\frac{{g}_{l_1+l_2}(X^{n+k})}{n})$,
where 
${g}_j(X^{n+k}):=\sum_{i=1}^n {g}_j(X_{i-1+k},\ldots, X_{i})$.
So, the expectation of the sample mean $\vec{Y}_{1,2}^n$ 
converges to $\vec{\eta}_{1,2}(0,\vec{\theta}_{2,3})$ as $n$ goes to infinity
with whatever initial distribution
when the true transition matrix is $W_{0,\vec{\theta}_{2,3}}$ and is irreducible \cite[Section 8.2]{HW-est}.
Now, we propose an estimator 
$\hat{\theta}_{2,3}(\vec{Y}_{1,2}^n) $
for $\vec{\theta}_{2,3}$
from the observed sample mean $\vec{Y}_{1,2}^n$ as follows.
\begin{align}
\hat{\theta}_{2,3}(\vec{Y}_{1,2}^n) :=
\argmin_{\vec{\theta}_{2,3}' \in \Theta_{2,3}}
\min_{\vec{\eta}_3' \in {\cal V}_3}
D\big(W_{\vec{\theta}(\vec{Y}_{1,2}^n,\vec{\eta}_3')}\big\|W_{0,\vec{\theta}_{2,3}'}\big).
\Label{eq6-9-1B}
\end{align}
Here, $\argmin_{\vec{\theta}_{2,3}' \in \Theta_{2,3}}$
is not unique in general.
In this case, we choose one of elements $\vec{\theta}_{2,3}' \in \Theta_{2,3}$ attaining the minimum.

In this case, still we need to care about the conditions C1, C2, and C3.
At least, we need to verify the Jacobi matrix condition given in Lemma \ref{6-13-thL}.
In the next two sections, we will investigate the equivalence problem when we apply the partial observation model to the $k-1$-memory joint distribution.

\section{Hidden Markov model}\Label{s8}
\subsection{Equivalence problem}\Label{s81}
To discuss estimation of the transition matrix of the hidden Markovian process,
we need to understand the relation between 
the stochastic behavior of the observed data and the transition matrix for hidden data.
Indeed, there is a possibility that
two different transition  matrices for hidden and observed variables
yield the same stochastic behavior for the observed variables.
Since such two transition matrices cannot be distinguished, we need to identify them
and consider that they are equivalent, in practice.
As a preparation to apply the result of Section \ref{s4} to hidden Markov process,
we need a parametrization by taking account into this equivalence relation for transition matrices.
That is, a proper parametrization without duplication is essential for 
the estimation of hidden Markov process.
For this aim, in this section, we review the result of the paper \cite{LEH}.

Given a transition matrix $W(x|x')$ from the hidden system ${\cal X}$ to itself
and
a transition matrix $V(y|x')$ from the hidden system ${\cal X}$ to the observed system ${\cal Y}$,
we say that the pair $(W,V)$
is a pair of transition matrices $(W,V)$ on ${\cal X}$ with ${\cal Y}$.
Since the pair $(W,V)$ gives a transition matrix  
$P^{2|1}_{XY}[(W,V)](x,y|x'y'):= W(x|x')V(y|x')$ on ${\cal X} \times {\cal Y}$,
it yields a hidden Markovian process with an observed variable $Y$ and 
a hidden variable $X$.
When the transition matrix $W$ is irreducible,
we call the pair $(W,V)$ of transition matrices irreducible.
When the transition matrix $W$ is ergodic,
and the image of the matrix $V$ is the vector space ${\cal V}_{{\cal Y}}$, 
the transition matrix $P^{2|1}_{XY}[(W,V)]$ is irreducible.
Hence, we call the pair $(W,V)$ of transition matrices irreducible in this case.

In the following, for simplicity, we identify ${\cal X}$ and ${\cal Y}$
with $\{1,\ldots, d\}$ and $\{1,\ldots, d_Y\}$, respectively.
That is, $|{\cal X}|=d$ and $|{\cal Y}|=d_Y$.
In this section, we assume the above irreducibility.
When the initial distribution is $P$ and the transition matrices are $W$ and $V$,
the probability that the sequence 
$(y_{k},y_{k-1}, \ldots , y_{1}) \in {\cal Y}^{k}$
is given as
\begin{align}
\sum_{ (x_0,x_1, \ldots , x_{k-1}) \in  \mathcal{X}^{k}}
V(y_{k}|x_{k-1}) W(x_{k-1}|x_{k-2}) V(y_{k-1}|x_{k-2})\cdots W(x_1|x_{0}) V(y_{1}|x_{0})P(x_0).
\Label{HFD}
\end{align}

For the following discussion, we employ the vector space 
${\cal V}_{{\cal X}}:= \{ v=(v_x)_{x \in {\cal X}}| v_x \in \bR\}$, 
i.e., the space ${\cal V}_{{\cal X}}$ is spanned by basis $\{e_x\}_{x \in {\cal X}}$.
We define the vector $V_{y} \in {\cal V}_{{\cal X}}$ as $V_y(x'):=V(y|x')$
and the matrix $D(v)$ on ${\cal V}_{{\cal X}}$ as the diagonal matrix whose diagonal entries are given by a vector $v \in {\cal V}_{{\cal X}}$.
Also, we define the vector ${V}_{*,x'} \in {\cal V}_{{\cal Y}}$ as ${V}_{*,x'}(y):= V(y|x')$ for 
a transition matrix $V(y|x')$ from ${\cal X}$ to ${\cal Y}$.
By using these notations and products of matrices on ${\cal V}_{{\cal X}}$,
the transition matrix  $P^k[(W,V)] $ is defined as
\begin{align*}
P^k[(W,V)] (y_k, \ldots, y_1|x'):=&
\sum_{x\in {\cal X}} (W D(V_{y_k}) W \cdots W D(V_{y_1})) (x|x') \\
=&
\sum_{x\in {\cal X}} (D(V_{y_k}) W \cdots W D(V_{y_1}))(x|x').
\end{align*}
Hence, the joint distribution \eqref{HFD} on ${\cal Y}^k$ is given as
\begin{align*}
P^{k}[(W,V)]\cdot P(y_k, \ldots, y_1)=
\sum_{x' \in {\cal X}} (P^k[(W,V)] (y_k, \ldots, y_1|x'))P(x'),
\end{align*}
where $\cdot$ express the matrix product on ${\cal V}_{{\cal X}}$.

The integer $k_{(W,V)}$ is defined to be the minimum integer $k_0$ to satisfy the condition
$\Ker P^{k_0}[(W,V)]=\cap_{k} \Ker P^{k}[(W,V)]$. 
For a distribution $P$ on ${\cal X}$, the subspace ${\cal V}^k(P)$ 
is defined as the subspace of ${\cal V}_{{\cal X}}/ \Ker P^{k_{(W,V)}}[(W,V)]$
spanned by $\{ [ W D(V_{y_k}) W \cdots W D(V_{y_1})P] | y_j \in \cY, j\le k\}$,
where $[v]$ expresses the element of 
${\cal V}_{{\cal X}}/ \Ker P^{k_{(W,V)}}[(W,V)]$ whose 
representative is $v \in {\cal V}_{{\cal X}}$.
Then, the integer $ k_{(P,(W,V))}$ is defined to be the minimum integer $k_1$ 
to satisfy the condition
$\cup_{k =1}^{\infty} {\cal V}^k(P) = {\cal V}^{k_1}(P)$.
Then, the following proposition gives the meaning of 
the integers $k_{(W,V)}$ and $ k_{(P,(W,V))}$.
\begin{proposition}[\protect{\cite[]{LEH}}]
The following conditions are equivalent for $P,W,V$ and $P',W',V'$.
\begin{description}
\item[(1)]
The relations 
$k_{(W,V)}=k_{(W',V')}$,
$k_{(P,(W,V))}=k_{(P',(W',V'))}$,
and
$P^{k}[(W,V)]\cdot P
=
P^{k}[(W',V')]\cdot P'$ 
hold with $k=k_{(W,V)}+k_{(P,(W,V))}+1$.

\item[(2)]
The relation
$P^{k}[(W,V)]\cdot P
=
P^{k}[(W',V')]\cdot P'$ 
holds with any integer $k$.

\item[(3)]
The relation 
$P^{k}[(W,V)]\cdot P
=
P^{k}[(W',V')]\cdot P'$ 
holds for $k=\max (k_{(W,V)},k_{(W',V')})
+\max(k_{(P,(W,V))},k_{(P',(W',V'))})+1$.

\end{description}
\end{proposition}

When the condition of the above proposition holds,
we say that the triple $P,W,V$ is equivalent to the triple $P',W',V'$.
For a triple $P,W,V$, 
the triple $P',W',V'$ is called the minimum realization of the triple $P,W,V$
when the triple $P',W',V'$ has the minimum set ${\cal X}$ among equivalent triples \cite{IAK}.
The numbers $k_{(W,V)}$ and $k_{(P,(W,V))}$ are upper bounded as follows.

\begin{proposition}[\protect{\cite{IAK},\cite[]{LEH}}]\Label{KTG}
The inequalities $k_{(W,V)}\le d$ and $k_{(P,(W,V))}\le d - \dim \Ker P^{k_{(W,V)}}[(W,V)]$
hold.
\end{proposition}

In a special case, $k_{(W,V)}$ 
is characterized as follows.

\begin{proposition}[\protect{\cite[]{LEH}}]\Label{KHG}
Given a pair of transition matrices $(W,V)$ on ${\cal X}$ with ${\cal Y}$ and a distribution $P$ on ${\cal X}$, 
we assume that the vectors $\{{V}_{*,x'}\}_{x' \in {\cal X}} $ are linearly independent.
Then, $\Ker P^1[(W,V)]=\{0\}$ and $k_{(W,V)}=1 $.
\end{proposition}

Although Proposition \ref{KHG} guarantees the relation 
$\Ker P^{k_{(W,V)}}[(W,V)]=\{0\}$ under a certain condition for the transition matrix $V$,
the condition is too strong because it does not hold when $d_Y < d$.
Even when $d_Y < d$, we can expect the relations
$\Ker P^{k_{(W,V)}}[(W,V)]=\{0\}$ and ${\cal V}^{k_{P,(W,V)}}(P)={\cal V}_{{\cal X}}$
under some natural condition.
The following lemma shows how frequently these conditions hold.

\begin{proposition}[\protect{\cite[]{LEH}}]\Label{2-15-B}
We fix a transition matrix $V$, and assume the existence of $y \in {\cal Y}$ such that $ V_y$ is not a scalar times of $u_{{\cal X}}$.
The relations
$\Ker P^{k_{(W,V)}}[(W,V)]=\{0\}$ and ${\cal V}^{k_{P,(W,V)}}(P)={\cal V}_{{\cal X}}$
hold almost everywhere with respect to $W$ and $P$.
Also, the relations $\Ker P^{k_{(W,V)}}[(W,V)]=\{0\}$ and ${\cal V}^{k_{P,(W,V)}}(P_{W})={\cal V}_{{\cal X}}$
hold almost everywhere with respect to $W$.
\end{proposition}

\subsection{Exponential family}
Next, to give a suitable parametrization,
we define the exponential family of pairs of transition matrices on ${\cal X}$ with ${\cal Y}$.
Firstly, we fix a par of transition matrices $(W,V)$ on ${\cal X}$, where $W$ is irreducible.
Then, we denote the stationary distribution of $W$ by $P_W$, and
denote the support of $(W,V)$ by
$({\cal X}^{2}\cup ({\cal Y}\times {\cal X}))_{(W,V)}:=
{\cal X}^{2}_{W} \cup ({\cal Y}\times {\cal X})_V$.
Also, we denote the linear space of real-valued functions 
$g=(g_a(x,x'),g_b(y,x'))$ defined on 
$({\cal X}^{2}\cup ({\cal Y}\times {\cal X}))_{(W,V)}$
by ${\cal G}_{(W,V)}$.
In this notation, 
for an element $(x,x')\in {\cal X}^{2}$, the function is given as $g_a(x,x')$,
and for an element $(y,x')\in {\cal Y}\times {\cal X}$, the function is given as $g_b(y,x')$.
In particular, we denote the linear space of real-valued functions 
$g=(g_a(x,x'),g_b(y,x'))$ defined on 
$({\cal X}^{2}\cup ({\cal Y}\times {\cal X}))_{(W,V)}$ with this constraint
$ \sum_{y \in {\cal Y}}V(y|x')g_b(y,x')=0$ for $x' \in {\cal X}$
by ${\cal G}_{0,(W,V)}$.
Additionally, we define the subspace 
${\cal N}_{(W,V)}^I$
as the subspace of functions with form $g_a(x,x')= f(x)-f(x')+c$ and $g_b(y,x')=0$.

To give the relation between $e$-representation and $m$-representation,
we define the linear map $(W,V)_*$ on 
${\cal G}_{(W,V)}$ as
\begin{align*}
((W,V)_*g)_a(x,x'):= g_a(x,x') W(x|x') , \quad
((W,V)_*g)_b(y,x'):= g_b(y,x') V(y|x')
\end{align*}
for $g\in {\cal G}_{(W,V)}$.
In the following, the function 
$g_a(x,x')$ is often written as a matrix $B$ on ${\cal V}_{{\cal X}}$,
and 
$g_b(y,x')$ is written as a collection of vectors $C=(C_y)_{y}$, which belong to 
${\cal V}_{{\cal X}}$.
That is, the map $(W,V)_{*}$ is rewritten as 
\begin{align}
(W,V)_{*}(B,C)= (W_* (B), (D(V_y) C_y)_y).
\end{align}
Hence, an element 
$(B,C) \in (W,V)_{*} {\cal G}_0
(({\cal X}^{2}\cup ({\cal Y}\times {\cal X}))_{(W,V)})$
satisfies $\sum_y C_y=0$.

Then, we define the subspaces of 
${\cal G}_{(W,V)}$ as
\begin{align}
{\cal L}_{1,W,V}^I :=&
\Big\{ 
(B,C)\in {\cal G}_{(W,V)}
 \Big|
B^T | u_{{\cal X}}\rangle =0,
\sum_{y \in {\cal Y}}C_{y}=0
\Big \} ,\nonumber \\
{\cal G}_{1,(W,V)}
:=& (W,V)_*^{-1}{\cal L}_{1,W,V}^I \nonumber \\
=&
 \Big\{(g_a, g_b) \Big| \forall x',\sum_x g_a(x, x')W(x|x') = 0,
\sum_y g_b(y, x')V (y|x') = 0\Big\}.
\end{align}

When functions $\hat{g}_1, \ldots, \hat{g}_l \in {\cal G}_{(W,V)}$
are linearly independent as elements of ${\cal G}_{(W,V)}
/{\cal N}^I_{(W,V)}$
for $\vec{\theta}:=(\theta^1, \ldots, \theta^l) \in \bR^l$,
we define the matrices 
\begin{align}
 V_{\vec{\theta}}(y|x')&:=
e^{\sum_{j=1}^l \theta^j \hat{g}_{j,b}(y,x')} V(y|x')
/\sum_{y'}e^{\sum_{j=1}^l \theta^j \hat{g}_{j,b}(y',x')} V(y'|x') \Label{NY1}\\
\overline{W}_{\vec{\theta}}(x|x')
&:=\sum_{y} e^{\sum_{j=1}^l \theta^j (\hat{g}_{j,a}(x, x') + \hat{g}_{j,b}(y,x'))}V(y|x') W(x|x'),\Label{NY2}
\end{align}
and denote the Perron-Frobenius eigenvalue by $\lambda_{\vec{\theta}}$.
Also, we denote the Perron-Frobenius eigenvector of 
the transpose $\overline{W}_{\vec{\theta}}^T$
by $\overline{P}^3_{\vec{\theta}}$.
Then, we define the transition matrix
\begin{align}
W_{\vec{\theta}}(x| x'):=
\lambda_{\vec{\theta}}^{-1} \overline{P}^3_{\vec{\theta}}(x)
\overline{W}_{\vec{\theta}}(x|x') \overline{P}^3_{\vec{\theta}}(x')^{-1}
\Label{NY3}
\end{align}
on ${\cal X}$.
We call $\{(W_{\vec{\theta}},V_{\vec{\theta}})\}_{\vec{\theta}}$ 
an exponential family of pairs of transition matrices generated by $\hat{g}_1, \ldots, \hat{g}_l $.
For the latter discussion, we define $l'$ be the dimension of the subspace
${\cal G}({\cal Y}) \cap <\hat{g}_1, \ldots, \hat{g}_l >$, where
$<\hat{g}_1, \ldots, \hat{g}_l >$ is the subspace spanned by $\hat{g}_1, \ldots, \hat{g}_l$
and ${\cal G}({\cal Y})$ is defined as the subspace 
of elements $g=(0,g_b(y,x'))$ in ${\cal G}_{(W,V)}$ such that 
$g_b(y,x')$ depends only on $y$.
Then, we reselect the generators as a basis $\hat{g}_1, \ldots, \hat{g}_l$ 
of the subspace $<\hat{g}_1, \ldots, \hat{g}_l >$
such that $\hat{g}_1, \ldots, \hat{g}_{l'} \in {\cal G}({\cal Y})$.
As explained in the paper \cite[Section 4.1]{LEH}, 
this exponential family can be regarded as a special case of 
exponential families of transition matrices on $\tilde{X}:={\cal X} \times {\cal Y}$.
The combination of \eqref{NY1}, \eqref{NY2}, and \eqref{NY3}
gives the form of $W_{\vec{\theta}}(x|x')V_{\vec{\theta}} (y|x')$ as
\begin{align*}
W_{\vec{\theta}}(x|x')V_{\vec{\theta}} (y|x') = 
\lambda_{\vec{\theta}}^{-1} \overline{P}^3_{\vec{\theta}}(x)
e^{\sum_j \theta^j (g_{j,a}(x,x')+g_{j,b}(y,x'))} 
V (y|x')W(x|x') \overline{P}^3_{\vec{\theta}}(x')^{-1}.
\end{align*}

\begin{example}[\protect{\cite[]{LEH}}]\Label{ERT}
As an example, we consider the full parameter model of pairs of transition matrices on ${\cal X}$ with ${\cal Y}$.
That is, we assume that the support
 $({\cal X}^{2}\cup ({\cal Y}\times {\cal X}))_{(W,V)}$ is 
 ${\cal X}^{2}\cup ({\cal Y}\times {\cal X})$
 and $W$ is irreducible.
The tangent space of the model is given by the space ${\cal L}_{1,W,V}^I $, whose dimension is 
$l:=d^2-d+d d_Y-d=d(d+d_Y-2)$.
In this case, we can easily find the generators as follows.
Here, we do not necessarily choose the generators from 
${\cal G}_{1,(W,V)}$.
That is, it is sufficient to choose them as elements of 
${\cal G}_{(W,V)}$.
Remember that ${\cal X}=\{1,\ldots, d\}$ and ${\cal Y}=\{1,\ldots, d_Y\}$.
To define $\hat{g}_j$ for $1 \le j \le l$,
we choose the functions 
$\hat{g}_{j,a}$ and $\hat{g}_{j,b}$ for $ 1\le j \le d_Y-1$, 
the functions 
$\hat{g}_{i(d_Y-1)+j,a}$ and 
$\hat{g}_{i(d_Y-1)+j,b}$ for $1\le i \le d-1, 1\le j \le d_Y-1$,
and the functions $\hat{g}_{d(d_Y-1)+(i-1)(d-1)+j,a}$ and $\hat{g}_{d(d_Y-1)+(i-1)(d-1)+j,b}$ 
for $1\le i \le d, 1\le j \le d-1$ as
\begin{align}
&\hat{g}_{j,a}(x,x') := 0\Label{LG1}, \quad
\hat{g}_{j,b}(y,x') :=\delta_{y,j} ,\quad
\hat{g}_{i(d_Y-1)+j,a}(x,x') := 0 \\
&\hat{g}_{i(d_Y-1)+j,b}(y,x') := \delta_{x',i}\delta_{y,j} \\
&\hat{g}_{d(d_Y-1)+(i-1)(d-1)+j,a}(x,x') := \delta_{x',i}\delta_{x,j} \\
&\hat{g}_{d(d_Y-1)+(i-1)(d-1)+j,b}(y,x') := 0 .\Label{LG2}
\end{align}
Then, the functions $\hat{g}_i=(\hat{g}_{i,a},\hat{g}_{i,b})$ are linearly independent.
We can parametrize the full model of pairs of transition matrices
by using these generators.
In particular, the first $d_Y-1$ functions belong to ${\cal G}({\cal Y})$.
That is, the maximum number $l'$ of observed generators is $d_Y-1$. 
Then, we have
the exponential family of pairs of transition matrices $(W_{\vec{\theta}},V_{\vec{\theta}})$,
which is generated by the above generators $\hat{g}_{1}, \ldots, \hat{g}_{l}$
at $(W,V)$.
\end{example}

The following proposition guarantees that
Example \ref{ERT} characterizes all pairs of transition matrices with full support.

\begin{proposition}[\protect{\cite[]{LEH}}]\Label{AGH}
\mh{If} the pair of transition matrices $(W,V)$ satisfies Condition
$({\cal X}^{2}\cup ({\cal Y}\times {\cal X}))_{(W,V)}
={\cal X}^{2}\cup ({\cal Y}\times {\cal X})$,
then the set $\{(W_{\vec{\theta}},V_{\vec{\theta}}) \}_{\vec{\theta}}$ defined in Example \ref{ERT}
equals the set of
pairs of transition matrices $(W',V')$ on ${\cal X}$ with ${\cal Y}$ satisfying the relation
 $({\cal X}^{2}\cup ({\cal Y}\times {\cal X}))_{(W',V')}
 ={\cal X}^{2}\cup ({\cal Y}\times {\cal X})$.
\end{proposition}

\subsection{Local equivalence}
However, as discussed in Subsection \ref{s81}, we cannot necessarily distinguish all the elements of 
the above exponential family because due to the equivalence problem.
Here, we discuss the equivalence problem among generators. 
For this aim, we define several subspaces of ${\cal G}_{(W,V)}$.
First, we define the subspace ${\cal L}_{P,W,V}^I$ as
\begin{align}
{\cal L}_{P,W,V}^I:=&
\Big\{ 
(B,C)\in {\cal G}_{(W,V)}
 \Big|
\hbox{Conditions \eqref{C1} and \eqref{C2} hold.}
\Big \},
\end{align}
where
Conditions \eqref{C1} and \eqref{C2} are defined as
\begin{align}
& B^T | u_{{\cal X}}\rangle = 0,\Label{C1} \\
& (W D(C_{y})+B D(V_y)) ({\cal V}^{k_{(P,(W,V))}}(P) +\Ker P^{k_{(W,V)}}[(W,V)] )
\subset  \Ker  P^{k_{(W,V)}}[(W,V)] 
\Label{C2}.
\end{align}
Here, ${\cal V}^{k_{(P,(W,V))}}(P) +\Ker P^{k_{(W,V)}}[(W,V)]$ expresses the subspace of 
${\cal V}_{{\cal X}}$ generated by $\Ker P^{k_{(W,V)}}[(W,V)]$ and the representatives of 
${\cal V}^{k_{(P,(W,V))}}(P)$ while
${\cal V}^{k_{(P,(W,V))}}(P)$ is a subspace of 
the quotient space ${\cal V}_{{\cal X}}/\Ker P^{k_{(W,V)}}[(W,V)]$.
Then, 
we define the subspaces as
\begin{align}
{\cal L}_{2,W,V}^I
:=&
\{ ([W,A] , C ) | 
A^T| u_{{\cal X}} \rangle=0,~
W[D(V_y),A]=W D(C_y)
\} \Label{YF1}.
\end{align}
When $W$ is invertible, we define ${\cal N}^I_{2,(W,V)} :=
(W,V)_*^{-1}({\cal L}_{2,\vec{W}}^I )$ and
${\cal N}^I_{P,(W,V)} :=
(W,V)_*^{-1}({\cal L}_{P,\vec{W}}^I )$.
By using ${\cal N}_{(W,V)}$ and these spaces,
the following proposition characterizes 
the equivalent conditions for generators in the asymptotic setting.
That is, it addresses what generators yield an observed infinitesimal change.

\begin{proposition}[\protect{\cite[]{LEH}}]\Label{27-3}
Assume that the transition matrix $W$ is irreducible.
Then, the following conditions are equivalent for 
functions $g_1, \ldots, g_l \in {\cal G}_{1,(W,V)}$ 
and a vector $\vec{a} \in \mathbb{R}^l$, where
$\{(W_{\vec{\theta}},V_{\vec{\theta}})\}_{\vec{\theta}}$ is the exponential family of 
pairs of transition matrices 
generated by the generators $\{g_i\}_{i=1}^l$.
\begin{description}
\item[(1)]
The function $\sum_{j=1}^l a^j g_j \in {\cal G}_{1,(W,V)}$
belongs to $
{\cal N}_{2,(W,V)}^I
+{\cal N}_{P_{W},(W,V)}^I$.

\item[(2)]
The relation $\sum_{j=1}^l a^j \frac{\partial }{\partial \theta^j}
P^{k}[(W_{\vec{\theta}},V_{\vec{\theta}})]
\cdot P_{W_{\vec{\theta}}} \Big|_{\vec{\theta}=0}
= 0$ holds for any positive integer $k$.

\item[(3)]
The relation $\sum_{j=1}^l a^j \frac{\partial }{\partial \theta^j}
P^{k_{(W,V)}+k_{(P_{(W,V)},(W,V))}+1}
[(W_{\vec{\theta}},V_{\vec{\theta}})]
\cdot P_{W_{\vec{\theta}}} \Big|_{\vec{\theta}=0} = 0$ holds.

\item[(4)]
The relation $\sum_{j=1}^l a^j \frac{\partial }{\partial \theta^j}
P^{k}[(W_{\vec{\theta}},V_{\vec{\theta}})]
\cdot P_{W_{\vec{\theta}}} \Big|_{\vec{\theta}=0} = 0$ holds
with a certain integer $k \ge k_{(W,V)}+k_{(P_{(W,V)},(W,V))}+1$.
\end{description}
\end{proposition}

Due to this theorem, under the above identification, 
the local and asymptotic equivalence class at $\theta$ is given as 
the quotient space 
${\cal G}_{(W,V)}/
({\cal N}_{(W,V)}^I
+{\cal N}_{P_{W},(W,V)}^I
+{\cal N}_{2,(W,V)}^I)$.
When the generators of our exponential family
are not linearly independent in the sense of the quotient space 
${\cal G}_{(W,V)}/
({\cal N}_{(W,V)}^I+{\cal N}_{P_{W},(W,V)}^I
+{\cal N}_{2,(W,V)}^I)$, the parametrization around $(W,V)$
does not express distinguishable information.
That is, the parametrization is considered to be redundant.
 
\begin{proposition}[\protect{\cite[]{LEH}}]\Label{VDT}
We assume that all the vectors $V_{*,x}$ are different
and the relations $ \Ker P^{k_{(W,V)}}[(W,V)]=\{0\}$ and ${\cal V}^{k_{(P,(W,V))}}(P)
={\cal V}_{{\cal X}}$ hold.
Then, the relations
$ \dim {\cal L}_{2,W,V}^I=\dim {\cal L}_{P,W,V}^I=0$ hold.
\end{proposition}

Since all the vectors $V_{*,x}$ are different almost everywhere with respect to $V$,
Propositions \ref{2-15-B} and \ref{VDT} guarantee that
the relation $ \dim {\cal L}_{2,W,V}^I= \dim {\cal L}_{P,W,V}^I=0$
holds almost everywhere with respect to $V,W$ in the exponential family $(W_{\vec{\theta}},V_{\vec{\theta}})$ of Example \ref{ERT}.
However, in several points, the dimensions of these spaces are not zero.
We call such points {\it singular points}.

\section{Estimation in hidden Markovian model}\Label{s6B}
\subsection{Application of Theorem \ref{6-13-th}}\Label{s6B1}
In this section, we apply the framework of a partial observation model given in Section \ref{s4}
to an exponential family of transition matrices
as $k-1$-memory transition matrices given in Section \ref{s4-5}.
For this aim, we prepare several notions for the application of Theorem \ref{6-13-th}.

When the transition matrix $W$ is ergodic
and the image of the matrix $V$ is the vector space ${\cal V}_{{\cal Y}}$, 
we can show that the transition matrix $P^{2|1}_{XY}[(W,V)]$ on ${\cal X}\times {\cal Y}$ is ergodic
in the similar way as Lemma \ref{STR3-2}.
Hence, we call the pair $(W,V)$ of transition matrices ergodic in this case.
In this section, we assume this property.
We prepare generators $\hat{g}_1, \ldots, \hat{g}_l$ defined on 
the support of $(W,V)$, i.e., $({\cal X}^{2}\cup ({\cal Y}\times {\cal X}))_{(W,V)}=
{\cal X}^{2}_{W} \cup ({\cal Y}\times {\cal X})_V$, where
$\hat{g}_i$ has a form $(\hat{g}_{i,a},\hat{g}_{i,b})$ for $i=1, \ldots, l$.
We assume that the initial $l'$ generators $\hat{g}_1, \ldots, \hat{g}_{l'}$ are given 
as functions of ${\cal Y}$, i.e., $\hat{g}_{i,a}=0$ 
and $\hat{g}_{i,b}(y,x)$ depends only on $y$ for $i=1, \ldots, l'$.
Hence, $\hat{g}_i$ can be regarded as an element of ${\cal G}({\cal Y})$ for
$i=1, \ldots, l'$.
Then, we address 
the exponential family of pairs of transition matrices $(W_{\vec{\theta}},V_{\vec{\theta}})$, which is generated 
by the generators $\hat{g}_1, \ldots, \hat{g}_l$ with $(W,V)$.
That is, we treat 
the pair of transition matrices $(W_{\vec{\theta}},V_{\vec{\theta}})$
as $k-1$-memory transition matrices 
$P_{XY}^{k|k-1}[(W_{\vec{\theta}},V_{\vec{\theta}})]$
on $\tilde{{\cal X}}:={\cal X} \times {\cal Y}$ in the following way. 
\begin{align*}
P_{XY}^{k|k-1}[(W_{\vec{\theta}},V_{\vec{\theta}})](x_k,y_{k}|x_{k-1},y_{k-1}, \ldots ,x_1, y_{1})
:= W_{\vec{\theta}}(x_k|x_{k-1})V_{\vec{\theta}}(y_{k}|x_{k-1}).
\end{align*}
In particular, we have
\begin{align*}
P_{XY}^{k|k-1}[(W,V)](x_k,y_{k}|x_{k-1},y_{k-1}, \ldots ,x_1, y_{1})
:= W(x_k|x_{k-1})V(y_{k}|x_{k-1}).
\end{align*}
Since the pair $(W,V)$ of transition matrices is ergodic,
Lemma \ref{STR3-2} guarantees that the $k-1$-memory transition matrix $P_{XY}^{k|k-1}[(W,V)]$ is ergodic.
When the distribution $P$ on ${\cal X}$ has the full support,
the support of $P^{k}[(W_{\vec{\theta}},V_{\vec{\theta}})]\cdot P$
does not depend on the parameter $\vec{\theta}$,
and equals that of $P^{k}[(W,V)]\cdot P$.
Hence, we denote the support by $ {\cal Y}^k_{(W,V)}$.
We focus on the function space ${\cal G}( {\cal Y}^k_{(W,V)})$ on $ {\cal Y}^k_{(W,V)}$,
and the subspace ${\cal N}( {\cal Y}^k_{(W,V)})$, which is defined as
the subspaces of functions of the form
$f(y_k, y_{k-1}, \ldots, y_2 )-f(y_{k-1}, y_{k-2}, \ldots, y_1 )+c$ for 
$(y_k, y_{k-1}, \ldots, y_1 ) \in {\cal Y}^k_{(W,V)}$.
In the following, 
we identify an element $g \in {\cal G}( {\cal Y})$
with an element $\bar{g}$ of ${\cal G}( {\cal Y}^k_{(W,V)})$ 
as $\bar{g}(y_k,\ldots, y_1):= g(y_k)$.
In this way, ${\cal G}( {\cal Y})$ can be regarded as a subspace of 
${\cal G}( {\cal Y}^k_{(W,V)})$.
Hence, the generators $\hat{g}_1, \ldots, \hat{g}_{l'}$
can be regarded as linearly independent as elements of ${\cal G}({\cal Y}^{k}_{(W,V)})/{\cal N}({\cal Y}^{k}_{(W,V)})$.
We choose observed generators $g_1, \ldots, g_{l_1}$
as elements of ${\cal G}({\cal Y}^{k}_{(W,V)})/{\cal N}({\cal Y}^{k}_{(W,V)})$
such that $g_1, \ldots, g_{l_1}, {\hat{g}}_1, \ldots, {\hat{g}}_{l'}$ 
is a basis of the space ${\cal G}({\cal Y}^{k}_{(W,V)})/{\cal N}({\cal Y}^{k}_{(W,V)})$.
Then, we choose $l_2:=l'$ and $l_3:= l-l'$,
and define $g_{l_1+i}$ to be ${\hat{g}}_{i}$ for $i=1, \ldots, l$.

Using the generators $g_1, \ldots, g_{l_1+l_2+l_3}$ and 
the $k-1$-memory transition matrix $P_{XY}^{k|k-1}[(W,V)]$ on ${\cal Y}\times{\cal X}$, 
we define the exponential family of
$k-1$-memory transition matrices $P_{XY|\vec{\theta}_1,\vec{\theta}_{2,3}}^{k|k-1}[(W,V)]$ on $\tilde{X}$,
where $\vec{\theta}_1=(\theta^1, \ldots,\theta^{l_1})$
and $\vec{\theta}_{2,3}=(\theta^{l_1+1}, \ldots,\theta^{l_1+l_2+l_3})$.
Since $P_{XY|0,\vec{\theta}_{2,3}}^{k|k-1}[(W,V)]
= P_{XY}^{k|k-1}[(W_{\vec{\theta}_{2,3}},V_{\vec{\theta}_{2,3}})]$,
we can apply the method in Section \ref{s4} to 
the exponential family of
$k-1$-memory transition matrices $P_{XY|\vec{\theta}_1,\vec{\theta}_{2,3}}^{k|k-1}[(W,V)]$ on $\tilde{\cal X}$,
for the estimation of the parameter of 
the exponential family of pairs of transition matrices $(W_{\vec{\theta}},V_{\vec{\theta}})$.

Here, to give a concrete form of $P_{XY|\vec{\theta} }^{k|k-1}[(W,V)]$ for
$\vec{\theta}=(\vec{\theta}_1,\vec{\theta}_{2,3})$,
we define a matrix $\tilde{W}^{k|k-1}_{\vec{\theta}}
$ on $\tilde{\cal X}^{k-1}$ as
\begin{align*}
&\tilde{W}^{k|k-1}_{\vec{\theta}}
(\tilde{x}_{k-1},\ldots, \tilde{x}_1| \tilde{x}_{k-1}',\ldots, \tilde{x}_1')
\\
:=&
\delta_{\tilde{x}_{k-2},\tilde{x}_{k-1}'}\ldots,
\delta_{\tilde{x}_{1},\tilde{x}_{2}'}
P_{XY}^{k|k-1}[(W,V)](\tilde{x}_{k-1}|\tilde{x}_{k-1}',\ldots, \tilde{x}_1')
e^{\sum_{i=1}^{l_1+l_2+l_3}\theta^i g_i(\tilde{x}_{k-1},\tilde{x}_{k-1}',\ldots, \tilde{x}_1') }.
\end{align*}
We denote the Perron-Frobenius eigenvalue 
and the Perron-Frobenius eigenvector of 
$(\tilde{W}^{k|k-1}_{\vec{\theta}})^T$
by $\lambda_{\vec{\theta}}$ and $P_{\vec{\theta}}^{k|k-1}$,
respectively.
Then, $P_{XY|\vec{\theta} }^{k|k-1}[(W,V)]$ is written as
\begin{align*}
&P_{XY|\vec{\theta}}^{k|k-1}[(W,V)](\tilde{x}_k|\tilde{x}_{k-1},\ldots, \tilde{x}_1)\\
=&
\lambda_{\vec{\theta}}^{-1}
P_{\vec{\theta}}^{k|k-1}(\tilde{x}_{k},\ldots, \tilde{x}_2)
\tilde{W}^{k|k-1}_{\vec{\theta}}
(\tilde{x}_{k},\ldots, \tilde{x}_2| \tilde{x}_{k-1},\ldots, \tilde{x}_1)
P_{\vec{\theta}}^{k|k-1}(\tilde{x}_{k-1},\ldots, \tilde{x}_1)^{-1}
\end{align*}
for $(\tilde{x}_{k},\tilde{x}_{k-1},\ldots, \tilde{x}_1)\in \tilde{\cal X}^k$.

\begin{example}\Label{ERT2}
As a typical example,
we apply the above discussion to 
the exponential family $(W_{\vec{\theta}},V_{\vec{\theta}})$ given in Example \ref{ERT}.
Hence, we assume the same assumption as Example \ref{ERT}.
Since $({\cal X}^{2}\cup ({\cal Y}\times {\cal X}))_{(W,V)}
={\cal X}^{2}\cup ({\cal Y}\times {\cal X})$,
the dimension of ${\cal G}({\cal Y}^{k})/{\cal N}({\cal Y}^{k})$ is
$d_Y^{k}-d_Y^{k-1} $.
Since $l_2= l'= d_Y-1$, we have $l_1= d_Y^{k}-d_Y^{k-1} -d_Y+1$ and $l_3=l-l_2=(d-1)(d-1+d_Y)$.
In the following, we choose $l_1$ functions
$g_{1}, \ldots, g_{l_1}$ as elements of 
${\cal G}({\cal Y}^k)$.

Now, we identify the set ${\cal Y}$ with the set $\{1, \ldots, d_Y\} $,
and the set ${\cal Y}^{k-1}$ with the set $\{1, \ldots, d_Y^{k-1}\} $.
Then, we define the functions 
$g_{i+(j-1) (d_Y^{k-1}-1) }$ on ${\cal Y}^k$
for $i=1, \ldots, d_Y^{k-1}-1$ and $j=1,\ldots, d_Y-1 $ as
\begin{align}
g_{i+(j-1) (d_Y^{k-1}-1) }(y_k, \vec{y}):=
\delta_{j,y_k}\delta_{i,\vec{y}},
\Label{eq2-10}
\end{align}
where $y_k \in {\cal Y}$ and $\vec{y} \in {\cal Y}^{k-1}$.
Then, we define $g_{l_1+1}, \ldots, g_{l_1+l_2+l_3}$ by
$ g_{l_1+i}:= \hat{g}_{i}$ for $i=1, \ldots, l_2+l_3$, where
$\hat{g}_{1}, \ldots \hat{g}_{l_2+l_3} $ are given in \eqref{LG1}--\eqref{LG2}.
\end{example}

Now, we apply the framework of 
a partial observation model of $k-1$-memory transition matrices
given in Subsection \ref{s5-2} to the case
when ${\cal X}$ is ${\cal Y} \times {\cal X}$,
the generators $g_1, \ldots, g_{l_1+l_2+l_3}$ are given in the above way,
and the $k-1$-memory transition matrix $W$ is given by 
the $k-1$-memory transition matrix $P_{XY}^{k|k-1}[(W,V)]$ on ${\cal Y}\times {\cal X}$.
While the projective Fisher information matrix $ \tilde{\mathsf{H}}_{\vec{\theta}_{2,3;o}}$ is defined in \eqref{FTE} and plays a central role in Section \ref{s4},
$ \tilde{\mathsf{H}}_{\vec{\theta}_{2,3;o}}^{k|k-1}$ expresses this matrix for 
the exponential family of
$k-1$-memory transition matrices $P_{XY|\vec{\theta}_1,\vec{\theta}_{2,3}}^{k|k-1}[(W,V)]$ on $\tilde{\cal X}$.
The aim of this application is the estimation of 
the parameter $\vec{\theta}_{2,3}$ to identify 
the $k-1$-memory transition matrix $P_{XY}^{k|k-1}[(W_{\vec{\theta}},V_{\vec{\theta}})]$
from the average of the observation on ${\cal Y}^k$.
In this application, we employ 
the estimator $\hat{\theta}_{2,3}(\vec{Y}_{1,2}^n) $ given in \eqref{eq6-9-1B}. 
However, in order that Theorem \ref{6-13-th} gives its asymptotic behavior,
we need Conditions {\bf C1}, {\bf C2}, and {\bf C3} in Subsection \ref{s4}.
At least, the projective Fisher information matrix 
$ \tilde{\mathsf{H}}_{\vec{\theta}_{2,3;o}}^{k|k-1}$ needs to be invertible.
For this issue, we have the following lemma by considering the condition 
{\bf B1} of Lemma \ref{6-13-thL}.

\begin{lemma}\Label{L-1-25-Y}
The following conditions are equivalent for a fixed positive integer $k$, $\vec{\theta}_0$,
and the generators $g_{l_1+1}, \ldots, g_{l_1+l_2+l_3} $.
\begin{description}
\item[E1]
We focus on the distribution $P^k[(W_{\vec{\theta}},V_{\vec{\theta}})]
\cdot P_{W_{\vec{\theta}}}$ on ${\cal Y}^{k}$.
The partial derivatives $\frac{\partial }{\partial \theta^j}P^k[(W_{\vec{\theta}},V_{\vec{\theta}})]
\cdot P_{W_{\vec{\theta}}}
|_{\vec{\theta}=\vec{\theta}_0}$ 
 are linearly independent for $j=l_1+1, \ldots, l_1+l_2+l_3$. 

\item[E2]
The Jacobi matrix condition, i.e.,
the condition of Lemma \ref{6-13-thL} holds at $\vec{\theta}_0$.
\end{description}
\end{lemma}

\begin{proof}
For $i=1, \ldots, l_1+l_2$ and $j=l_1+1, \ldots, l_1+l_2+l_3$,
\begin{align}
 A(\vec{\theta}_{0})_{i,j}=
\left.\frac{\partial \eta_i(0,\vec{\theta})}{\partial \theta^{j}}
\right|_{\vec{\theta}_{2,3}=\vec{\theta}_{0}}
=&
\left.\frac{\partial }{\partial \theta^{j}}
\sum_{\vec{y} \in {\cal Y}^k}
g_{i}(\vec{y}) P^k[(W_{\vec{\theta}},V_{\vec{\theta}})]
\cdot P_{W_{\vec{\theta}}} (\vec{y}) 
\right|_{\vec{\theta}=\vec{\theta}_{0}}\nonumber \\
=&
\sum_{\vec{y} \in {\cal Y}^k}
g_{i}(\vec{y})\left.\frac{\partial }{\partial \theta^{j}}
 P^k[(W_{\vec{\theta}},V_{\vec{\theta}})]
\cdot P_{W_{\vec{\theta}}} (\vec{y}) 
\right|_{\vec{\theta}=\vec{\theta}_{0}}.
\end{align}
Therefore, since $g_1, \ldots, g_{l_1+l_2}$ span 
the space ${\cal G}({\cal Y}^{k}_{(W,V)})/{\cal N}({\cal Y}^{k}_{(W,V)})$,
due to Lemmas \ref{STR2} and \ref{STR},
the linear independence of the vectors
$(A(\vec{\theta}_{0})_{i,l_1+1})_i, \ldots,(A(\vec{\theta}_{0})_{i,l_1+l_2+l_3})_i$
is equivalent to 
the linear independence of 
the partial derivatives 
$\frac{\partial }{\partial \theta^{l_1+1}}P^k[(W_{\vec{\theta}},V_{\vec{\theta}})]
\cdot P_{W_{\vec{\theta}}}
|_{\vec{\theta}=\vec{\theta}_0}, \ldots$, 
$\frac{\partial }{\partial \theta^{l_1+l_2+l_3}}P^k[(W_{\vec{\theta}},V_{\vec{\theta}})]
\cdot P_{W_{\vec{\theta}}}
|_{\vec{\theta}=\vec{\theta}_0}$. 
Hence, the condition {\bf E1} holds if and only if
the rank of $A(\vec{\theta}_{0})$ is $l_2+l_3$, which is equivalent to the condition {\bf E2}.
\end{proof}

Therefore, when $\Theta$ is an open set, and the map 
$\theta \mapsto P^k[(W_{\vec{\theta}},V_{\vec{\theta}})]\cdot P_{W_{\vec{\theta}}}$
is one-to-one, and the condition in Lemma \ref{L-1-25-Y} holds,
Conditions {\bf C1}, {\bf C2}, and {\bf C3} in Subsection \ref{s4} hold.
Hence, the performance of the estimator given in \eqref{eq6-9-1B} is characterized by Theorem \ref{6-13-th}.
That is, this estimator works well. 
Unfortunately, the condition in Lemma \ref{L-1-25-Y} does not always hold
in Example \ref{ERT2}.
Fortunately, due to Propositions \ref{27-3} and \ref{VDT}, we have the following theorem, which guarantees 
this condition under Example \ref{ERT2} with some natural assumptions.
Then, as the precise statement of Theorem \ref{PT2}, we obtain the following theorem.

\begin{theorem}\Label{T3-29}
Let $\{(W_{\vec{\theta}},V_{\vec{\theta}})\}_{\vec{\theta} \in \bR^{l_2+l_3}}$
be the exponential family of pairs of transition matrices given in Example \ref{ERT2}.
We choose an open subset $\Theta_{2,3}\subset \bR^{l_2+l_3}$ to satisfy the conditions {\bf C1} and {\bf C2} given in 
Section \ref{s4}.
\mh{If} the following conditions hold at any element $\vec{\theta}=\vec{\theta}_{2,3;o} \in  \Theta_{2,3}$,
then the projective Fisher information matrix 
$ \tilde{\mathsf{H}}_{\vec{\theta}_{2,3;o}}^{k|k-1}$ is invertible.
Hence, 
when the true parameter is $\vec{\theta}_{2,3;o} $,
the random variable 
$\hat{\theta}_{2,3}(\vec{Y}_{1,2}^n) - \vec{\theta}_{2,3;o}$
asymptotically obeys
the Gaussian distribution with the covariance matrix 
$\frac{1}{n}
( \tilde{\mathsf{H}}_{\vec{\theta}_{2,3;o}}^{k|k-1})^{-1}$.

\begin{description}
\item[F1]
$\Ker P^{k_{(W,V)}}[(W_{\vec{\theta}},V_{\vec{\theta}})]=\{0\}$ and 
${\cal V}^{k_{P,(W_{\vec{\theta}},V_{\vec{\theta}})}}(P)={\cal V}_{{\cal X}}$.

\item[F2]
All the vectors $V_{\vec{\theta},*,x}$ are different.

\item[F3]
The inequality $k \ge k_{(W_{\vec{\theta}},V_{\vec{\theta}})}
+k_{(P_{(W_{\vec{\theta}},V_{\vec{\theta}})},(W_{\vec{\theta}},V_{\vec{\theta}}))}+1$ holds.

\end{description}
\end{theorem}

\begin{proof}
Proposition \ref{VDT} and Conditions {\bf F1} and {\bf F2} guarantee the relations
$ \dim {\cal L}_{2,W,V}^I=\dim {\cal L}_{P,W,V}^I=0$.
Hence, $g_{l_1+1}, \ldots, g_{l_1+l_2+l_3}$ are linearly independent at 
any element $\vec{\theta}_{2,3;o}\in  \Theta_{2,3}$
in the sense of the quotient space
${\cal G}_{(W,V)}/
({\cal N}_{(W,V)}
+{\cal N}_{P_{W},(W,V)}
+{\cal N}_{2,(W,V)})$.
The combination of this fact, the condition {\bf F3}, and Proposition \ref{27-3} implies the condition {\bf E1} of Lemma \ref{L-1-25-Y}
at any element $\vec{\theta}_{2,3;o}\in  \Theta_{2,3}$.
Hence Lemmas \ref{6-13-thL} and \ref{L-1-25-Y} 
show that the projective Fisher information $ 
 \tilde{\mathsf{H}}_{\vec{\theta}_{2,3;o}}^{k|k-1}$
is invertible
at any element $\vec{\theta}_{2,3;o}\in  \Theta_{2,3}$.
Due to Theorem \ref{6-13-th},
the random variable 
$\hat{\theta}_{2,3}(\vec{Y}_{1,2}^n) - \vec{\theta}_{2,3;o}$
asymptotically obeys
the zero-mean Gaussian distribution of covariance matrix 
$\frac{1}{n}
( \tilde{\mathsf{H}}_{\vec{\theta}_{2,3;o}}^{k|k-1})^{-1}$
when the true parameter is $\vec{\theta}_{2,3;o} $.
\end{proof}

Now, we freely choose the transition matrix $W$ on ${\cal X}$ and the transition matrix $V$ from ${\cal X}$ to ${\cal Y}$.
Since Conditions {\bf F2} and {\bf F3} hold almost everywhere, 
due to Proposition \ref{2-15-B},
all the Conditions {\bf F1}, {\bf F2}, and {\bf F3} hold almost everywhere.
Hence, when our model is given by the exponential family given in Example \ref{ERT2},
the estimator 
$\hat{\theta}_{2,3}(\vec{Y}_{1,2}^n) $ given in \eqref{eq6-9-1B}
estimates the true parameter except for measure-zero sets.

\subsection{Special cases}
Although we have discussed how to apply our result in the special case in Example \ref{ERT2},
we analyze the more details in the more special case, which might be helpful for the readers to understand the mathematical structure discussed in this paper.  
In this subsection, we employ the parametrization given in Examples \ref{ERT} and \ref{ERT2}.

\subsubsection{$d=2$ and $d_Y=2$}
When $d_Y=2$, we choose $W$ and $V$ as
$W=V=
\left(
\begin{array}{cc}
\frac{1}{2} & \frac{1}{2} \\
\frac{1}{2} & \frac{1}{2} 
\end{array}
\right)$.
As discussed in \cite[]{LEH}, the subset of singular elements equals the set of non-memory cases, 
which can be characterized as $\theta^2=0$ or $\theta^3=\theta^4=0$.
This case is simply described by a binomial distribution and denoted by the model $\theta^{1\times 2}$. 
Hence, the set of non-singular elements are given as
$\bR \times (\bR\setminus \{0\})\times (\bR^2\setminus \{(0,0)\})$,
which can be divided into two connected components
$\theta^{2\times 2}:=\bR \times (0,\infty)\times (\bR^2\setminus \{(0,0)\})$ 
and
$\bR \times (-\infty,0)\times (\bR^2\setminus \{(0,0)\})$.
Each connected component has a one-to-one correspondence to 
non-singular elements divided by the equivalence class.

When we have a possibility for two models $\theta^{2\times 2}$ and $\theta^{1\times 2}$,
we need to choose one of the two models.
For this aim, employing Theorem \ref{6-13-thB},
we apply the $\chi^2$-test for this choice with a certain significance level $\alpha>0$.
That is, when the smaller model $\theta^{1\times 2}$ passes the $\chi^2$-test we adopt it.
Otherwise, we adopt the larger one.

\subsubsection{$d=2$ and $d_Y \ge 3$}
When $d_Y\ge 3$, we choose $W$ and $V$ as
$W=
\left(
\begin{array}{cc}
\frac{1}{2} & \frac{1}{2} \\
\frac{1}{2} & \frac{1}{2} 
\end{array}
\right),\quad
V=
\left(
\begin{array}{cc}
\frac{1}{d_Y} & \frac{1}{d_Y} \\
\frac{1}{d_Y} & \frac{1}{d_Y} \\
\vdots & \vdots \\
\frac{1}{d_Y} & \frac{1}{d_Y} 
\end{array}
\right)$.
As discussed in \cite[]{LEH}, the subset of singular elements equals the set of non-memory cases, 
which can be characterized as $\theta^{d_Y}=\cdots =\theta^{2d_y-2}=0$ or 
$\theta^{2d_Y-1}=\theta^{2d_Y}=0$.
We denote this model by $\theta^{1 \times d_Y}$.
Let $\theta^{NM}$ be the set 
$(\bR^{d_Y-1} \times \{(0,\ldots,0)\}\times \bR^{2})
\cup (\bR^{2d_Y-2} \times \{(0,0)\})$.
Then, the set $\bR^{2 d_Y} \setminus \theta^{NM}
=\bR^{d_Y-1}\times (\bR^{d_Y-1}\setminus \{(0,\ldots,0)\})
\times (\bR^{2}\setminus \{(0,0)\})$ 
equals to the set of non-singular elements.
However, it is impossible to divide the set $\bR^{2 d_Y} \setminus \theta^{NM}$
into components satisfying the following conditions.
(1) Each component is an open set.
(2) Each component gives a one-to-one parametrization for non-singular elements.
This is because the set $\bR^{2 d_Y} \setminus \theta^{NM}$ is connected.
We denote this model by $\theta^{2 \times d_Y}$.

Hence, we need to adopt a duplicated parametrization because the parametric space is needed to be open due to the assumption of Theorem \ref{6-13-th}.
Adopting this parametrization,
we apply the em-algorithm to find the estimate $\hat{\theta}_{2,3}(\vec{Y}_{1,2}^n)$ given in \eqref{eq6-9-1B}.
In this case, we find one parameter $\hat{\theta}_{2,3}(\vec{Y}_{1,2}^n)$ of parameters to attain the minimum \eqref{eq6-9-1B}.
We measure the error by the difference between the estimate 
$\hat{\theta}_{2,3}(\vec{Y}_{1,2}^n)$ and the parameter that is close to the estimate among parameters to generate the true process.
In this case, the asymptotic behavior is characterized by Theorem \ref{6-13-th}.
Further, when we have a possibility for two models $\theta^{2\times d_Y}$ and $\theta^{1\times d_Y}$,
we can employ the $\chi^2$-test for the model selection.

\subsubsection{$d \ge 3$}
When $d\ge 3$, it is not so easy to characterize all the singular points.
Indeed, given $d_Y$, it is not trivial to identify the number $d$ of hidden states.
When $d$ is fixed, due to Proposition \ref{AGH}, 
the model given in Examples \ref{ERT} 
covers all the transition matrices of $d$ hidden states with full support.
We denote this model by $\theta^{d\times d_Y}$. 
Similar to the above case, we employ the parametric space $\bR^{d(d+d_Y-2)}$,
and use the estimate $\hat{\theta}_{2,3}(\vec{Y}_{1,2}^n)$ by using the em-algorithm.
It is not clear how large the error of the estimation is
when the true parameter is a singular point.
But, if the pair of transition matrices parametrized by the estimate is close to the true one,
this method works well.

When the number $d$ of hidden states is unknown, we need to employ the model selection.
In this case, 
we apply the $\chi^2$-test for each model $\theta^{d\times d_Y}$ with a certain significance level $\alpha>0$  in the sense of Theorem \ref{6-13-thB}.
Then, we adopt the model with the minimum $d$ among passing models. 
Indeed, when the true parameter belongs to the set of singular points except for the submodels,
Theorem \ref{6-13-thB} cannot be applied.
Hence, we cannot guarantee the quality of this model selection in this exceptional case.

\section{Conclusion}\Label{s9}
We have formulated estimation of hidden Markov process 
by using the information geometrical structure 
(e.g., the exponential family, the natural parameter, the expectation parameter, 
divergence, projective Fisher information matrix, 
the Pythagorean theorem, and the em-algorithm) of transition matrices.
Since this geometrical structure does not change according to the number $n$ of observations,
the calculation complexity of our em-algorithm
does not depend on the number $n$ of observations.
We have also derived the asymptotic evaluation of the error of our estimator.

For this discussion, we have first formulated the partial observation model of the
Markovian process.
Under this model, we have formulated an estimator by using the em-algorithm based on 
the geometry of transition matrices.
Then,  we have asymptotically evaluated the error of the estimator
by using the projective Fisher information.
To apply these results to the estimation of the transition matrix of the hidden Markovian process, 
we have employed the result for the equivalence problem in terms of the tangent space by another paper \cite{LEH}.
Then, we have discussed the application of the results of a partial observation model
to an exponential family of pairs of transition matrices.
In particular, we have given a concrete parametrization of this application in a typical case as Example \ref{ERT2}.

In summary, once we obtain the sample mean of the generators $g_1, \ldots, g_{l_1+l_2}$ given in Example \ref{ERT2},
we can calculate the estimator $\hat{\theta}_{2,3}(\vec{Y}_{1,2}^n) $ given in \eqref{eq6-9-1B},
whose calculation complexity does not depend on $n$ and depends on $k$.
Since the estimation error is asymptotically characterized by 
the projective Fisher information matrix $ 
 \tilde{\mathsf{H}}_{\vec{\theta}_{2,3;o}}^{k|k-1}$,
it is sufficient to choose $k$ such that the projective Fisher information matrix 
$ \tilde{\mathsf{H}}_{\vec{\theta}_{2,3;o}}^{k|k-1}$ is not so small.
That is, we do not need to increase $k$ as $n$ increases.
At least, Theorem \ref{T3-29} guarantees that 
the matrix $ \tilde{\mathsf{H}}_{\vec{\theta}_{2,3;o}}^{k|k-1}$
is invertible when $k$ is sufficiently large.
Hence, when we apply our method $\hat{\theta}_{2,3}(\vec{Y}_{1,2}^n) $ given in \eqref{eq6-9-1B}
with the parametrization given in Example \ref{ERT2},
the calculation complexity is $O(1)$
once we obtain the averages $Y_i^n:=\sum_{j=0}^{n-1} \frac{1}{n}g_i(Y_{i+1}, \ldots,Y_{i+k}) $ with $i=1, \ldots, l_1+l_2$.
This method employs an optimization 
over the probability distribution on the state space whose size depends 
only on $k$ and the sizes $d$ and $d_Y$ of the state spaces ${\cal X}$ and ${\cal Y}$.

\mh{Therefore, we can compare our method
with the existing methods, e.g., Baum-Welch algorithm \cite{Baum,KGMS} and the 
conventional em-algorithm \cite{Amari-em,FN,KGMS} as follows.
Both methods contain the optimization based on several steps of iterations.
In the existing methods, each iteration needs
space complexity $O(n )$ and time complexity $O(n )$.
Also, the required number of iterations for the convergence depends on $n$,
and the evaluation of its dependence is a difficult problem and has been studied actively \cite{KGMS}.
In contrast, 
in our method, each iterations needs only space complexity $O(1 )$ and time complexity $O(1 )$
because our optimization problem does not change dependently of the number $n$ of samples.
Also, the required number of iterations for the convergence does not depend on $n$.
In fact, while these numbers depend on $k$ and $k$ depends on $d$ and $d_Y$, 
$k$ is a constant number and is independent of $n$.
Although these existing methods do not need the preparation stage,
as the preparation for the optimization,
our method needs to calculate the sample mean of the generators $g_1, \ldots, g_{l_1+l_2}$.
This preparation stage
requires the calculation complexity $O(n)$, but its space complexity is $O(\log n)$
because we only need to keep the histogram of past observations.
We do not need to repeat this preparation step.
That is, our method has an advantage for space complexity and time complexity
as summarized in Table \ref{hikaku1}.}

\begin{table}[t]
\caption{Comparison between our method and existing method}
\label{hikaku1}
\begin{center}
\begin{tabular}{|c|c|c|c|c|c|}
\hline
&Time & Space & Number &Time & Space \\
&complexity & complexity &of  &complexity  &  complexity   \\
\cline{2-3}\cline{5-6}
& \multicolumn{2}{|c|}{of each iteration}&iteration & \multicolumn{2}{|c|}{of non-iterative part}  \\
\hline
Existing method &\multirow{2}{*}{$O(n)$}&\multirow{2}{*}{$O(n)$}&\multirow{2}{*}{ \# }  & \multirow{2}{*}{$O(1)$}& \multirow{2}{*}{$O(1)$}     \\
\cite{Baum,Amari-em,FN,KGMS} & & & &&\\
\hline 
Our method & $O(1)$&$O(1)$&$O(1)$&$O(n)$&$O(\log n)$ \\
\hline
\end{tabular}
\end{center}
\#:  
The required number of iterations depends on $n$.
Its evaluation is very complicated and is actively studied.
It is one of the main issues in the study of 
conventional em-algorithm \cite[Section 2]{KGMS}.
\end{table}

\mh{Although the above discussion shows the advantage of our method over the existing methods,
to clarify the advantage more explicitly,
it is needed to numerically compare our method with the existing methods.
For this comparison, we need to implement the proposed algorithm,
which needs various knowledges of numerical optimizations.
Since these optimizations require
special skills for numerical calculation,
this implementation and the above numerical comparison
are beyond the scope of this paper, and should be done as an independent research.
These numerical analyses are very important future studies.}

Here, we need to mention that there are several singular points in this model.
Our asymptotic evaluation of the estimation error does not work when the true parameter is close to the singular point.
In this situation, the set of singular points is considered as another model.
Hence, it is needed to apply model selection.
Although there are several other methods to select our model, e.g., AIC and BIC, (MDL), 
these usually assume that there is no singularity.
In this paper, to discuss the model selection even with singularity, 
we propose to use the $\chi^2$-test as Theorem \ref{6-13-thB}.
However, to apply our method, we need to classify the singular points in the model of hidden Markov process.
Since this problem is too difficult, we could not discuss this problem. 
This is an interesting future problem.

\section*{Acknowledgment}
The author is very grateful to 
Professor Takafumi Kanamori, Professor Vincent Y. F. Tan, and Dr. Wataru Kumagai
for helpful discussions and comments.
The works reported here were supported in part by 
Guangdong Provincial Key Laboratory (Grant No. 2019B121203002),
the JSPS Grant-in-Aid for Scientific Research 
(B) No. 16KT0017, (A) No.17H01280, 
the Okawa Research Grant
and Kayamori Foundation of Informational Science Advancement.

\appendix

\section{Implementation of em-algorithm}\Label{ASImp}
It is not trivial to implement E-step and M-step because it is not easy to calculate the derivatives of the Perron-Frobenius eigenvalue.
Although the implementation of the em-algorithm based on Bregman divergence and its related algorithm
were discussed in \cite{FN,MTKE},
we discuss this problem in a different way.
Fortunately, we can avoid the calculations of the derivatives as follows.
In the M-step, we find $\vec{\theta}_{2,3;j}$ from $\vec{\eta}_{3;j}$ as follows.
Using the formula \eqref{eq1-28-2}, we calculate $\vec{\theta}(\vec{Y}_{1,2}^n,\vec{\eta}_{3;j})$.
Then, by using \eqref{1-1}, 
$\vec{\theta}_{2,3;j}$ is given as
\begin{align}
\argmin_{\vec{\theta}}
\phi(\vec{\theta})
-\sum_{i=l_1+1}^{l_1+l_2} 
\theta^{i}
(\vec{Y}_{1,2}^n)_i
-
\sum_{i=l_1+l_2+1}^{l_1+l_2+l_3} 
\theta^{i}
(\vec{\eta}_{3;j})_i.
\end{align}
These calculations can be done by 
derivative-free optimization algorithms \cite{CSV,MKK} represented by Nelder-Mead method \cite{NM}.
A derivative-free optimization algorithm maximizes a concave function without calculating the derivative only with calculating the outcomes with several inputs.

In the E-step, we find $\vec{\eta}_{3;j}$ from $\vec{\theta}_{3;j-1}$ as follows.
By using the formula \eqref{5-28-1} and \eqref{eq1-28-1}, 
$\vec{\eta}_{3;j}$ is given as
\begin{align}
\argmin_{\vec{\eta}_3'}
\max_{\vec{\theta}}
\sum_{i=l_1+1}^{l_1+l_2} 
(\vec{\theta} - \vec{\theta}_{2,3;j-1})^i (\vec{Y}_{1,2}^n)_i
+
\sum_{i=l_1+l_2+1}^{l_1+l_2+l_3} 
(\vec{\theta} - \vec{\theta}_{2,3;j-1})^i (\vec{\eta}_3')_i
-\phi(\vec{\theta}).
\end{align}
Since $(\vec{\theta} - \vec{\theta}_{2,3;j-1})^i (\vec{Y}_{1,2}^n)_i
+\sum_{i=l_1+l_2+1}^{l_1+l_2+l_3} 
(\vec{\theta} - \vec{\theta}_{2,3;j-1})^i (\vec{\eta}_3')_i
-\phi(\vec{\theta})$
is concave for $\vec{\theta} $ and is linear for $\vec{\eta}_3'$,
using the minimax theorem \cite[Chap. VI Prop. 2.3]{ET}, we have
\begin{align}
&\min_{\vec{\eta}_3'}
\max_{\vec{\theta}}
\sum_{i=l_1+1}^{l_1+l_2} 
(\vec{\theta} - \vec{\theta}_{2,3;j-1})^i (\vec{Y}_{1,2}^n)_i
+
\sum_{i=l_1+l_2+1}^{l_1+l_2+l_3} 
(\vec{\theta} - \vec{\theta}_{2,3;j-1})^i (\vec{\eta}_3')_i
-\phi(\vec{\theta}) \nonumber \\
=&\max_{\vec{\theta}}
\min_{\vec{\eta}_3'}
\sum_{i=l_1+1}^{l_1+l_2} 
(\vec{\theta} - \vec{\theta}_{2,3;j-1})^i (\vec{Y}_{1,2}^n)_i
+
\sum_{i=l_1+l_2+1}^{l_1+l_2+l_3} 
(\vec{\theta} - \vec{\theta}_{2,3;j-1})^i (\vec{\eta}_3')_i
-\phi(\vec{\theta}) \nonumber \\
=&
\max_{\vec{\theta}: \vec{\theta}^i = \vec{\theta}_{2,3;j-1}^i (i \ge l_1+ l_2+1)}
\sum_{i=l_1+1}^{l_1+l_2} 
(\vec{\theta} - \vec{\theta}_{2,3;j-1})^i (\vec{Y}_{1,2}^n)_i
-\phi(\vec{\theta}) .\Label{eq1-28-5}
\end{align}
The RHS of \eqref{eq1-28-5} can be easily calculated.
However, the maximization of the RHS of \eqref{eq1-28-5} cannot directly give the value of $\vec{\eta}_{3,j} $.
To calculate this value, we calculate 
\begin{align}
\vec{\theta}_{j}:=
\argmax_{\vec{\theta}: \vec{\theta}^i = \vec{\theta}_{2,3;j-1}^i(i \ge l_1+ l_2+1)}
\sum_{i=l_1+1}^{l_1+l_2} 
(\vec{\theta} - \vec{\theta}_{2,3;j-1})^i (\vec{Y}_{1,2}^n)_i
-\phi(\vec{\theta}) .
\end{align}
So, we have
$(\vec{Y}_{1,2}^n ,\vec{\eta}_{3;j})_i
=\frac{\partial \phi}{\partial \theta^i}(\vec{\theta}_{j})$.
Fortunately, we can avoid to calculate the derivatives as follows.
Since the stationary distribution is given as the Perron-Frobenius eigenvector,
this value can be calculated as the expectation of $g_j$ under the stationary distribution.

\section{Proof of Lemma \protect{\ref{L9-29}}}\Label{AS2}
\noindent{\it Step 1:\quad}
We will show that
\begin{align}
{\cal V}_{{\cal X}}= {\cal N}_W:=
\{ (W_* g)^T| u_{{\cal X}}\rangle | g \in {\cal N}({\cal X}^{2}_W) \}.\Label{9-29-1}
\end{align}
For this purpose, we will show that 
the function $|f\rangle- \langle f| P_W\rangle |u_{{\cal X}}\rangle$ 
belongs to the RHS of \eqref{9-29-1} for any function $f$.
Since 
\begin{align*}
&( ( {W}_* (- |f\rangle \langle u_{{\cal X}}| +|u_{{\cal X}}\rangle \langle f|  )^T 
|u_{{\cal X}}\rangle)_{x'}\nonumber \\
=&\sum_{x}W(x|x')(-f(x)+f(x')) 
=f(x')-\sum_{x}f(x)W(x|x')
=f(x')-(W^T |f\rangle)_{x'},
\end{align*}
the vector $ |f\rangle- W^T |f\rangle$ belongs to the set ${\cal N}_{W}$.
So, 
$ |f\rangle- W^T| f\rangle +W^T| f\rangle- W^T W^T |f\rangle 
=| f\rangle-W^T W^T |f\rangle$ belongs to the set ${\cal N}_{W}$.
Repeating this procedure, we see that
$|f\rangle - (W^T)^n |f\rangle$ belongs to the set ${\cal N}_{W}$.
Since $\lim_{n \to \infty}\frac{1}{n}\sum_{i=1}^n \sum_{x}f(x)W^i(x|x')= \sum_{x}f(x)P_{ W}(x) 
=\langle f| P_{ W} \rangle $ for any $x' \in {\cal X}$,
we have $\lim_{n \to \infty }\frac{1}{n}\sum_{i=1}^n
(|f\rangle - (W^T)^i |f\rangle)
=|f\rangle- \langle f| P_{ W}\rangle |u_{{\cal X}}\rangle$, i.e., 
$|f\rangle- \langle f| P_{ W}\rangle |u_{{\cal X}}\rangle$ belongs to the set ${\cal N}_{W}$.
Since 
\begin{align}
(( {W}_*  |u_{{\cal X}}\rangle \langle  u_{{\cal X}}| )^T| u_{{\cal X}}\rangle)_{x'}
 = \sum_{x}W(x|x')=1= |u_{{\cal X}}\rangle_{x'},
\end{align}
$ |u_{{\cal X}}\rangle $ belongs to the set ${\cal N}_{W}$.
Thus, any function $f$ belongs to the set ${\cal N}_{W}$, which implies 
\eqref{9-29-1}.

\noindent{\it Step 2:\quad}
It is sufficient to show that 
for $g' \in {\cal G}( {\cal X}^{2}_{W})$, there exists 
$g \in{\cal G}_{1,W}( {\cal X}^{2}_{W})$ such that 
$[g']=[g]$.
Due to \eqref{9-29-1}, we can choose an element $g'' \in {\cal N}( {\cal X}^{2}_{W})$
such that $(W_* g' )^T| u_{{\cal X}}\rangle  =( W_* g'')^T| u_{{\cal X}}\rangle $.
Hence, $g:= g'-g''$ belongs to ${\cal G}_{1,W}( {\cal X}^{2}_{W})$
because 
\begin{align*}
 W_* g^T | u_{{\cal X}}\rangle
= (W_* g')^T | u_{{\cal X}}\rangle
- (W_* g'')^T | u_{{\cal X}}\rangle 
=0.
\end{align*}
Since $g'' \in {\cal N}( {\cal X}^{2}_{W})$, 
we find that $[g']=[g]$.
\endproof

\section{Proofs of statements in Section \ref{s4}}\Label{s4-3}
To show Lemma \ref{6-13-thL}, we prepare the following lemma.
\begin{lemma}\Label{6-13-thC}
Let $K$ be a strictly positive definite matrix on ${\cal V}_{1,2,3}$, i.e., $K$ does not have zero eigenvalue.
Then, $P_{1,2} K P_{1,2}^T 
- (P_{1,2} K P_3^T)(P_3 KP_3^T)^{-1}(P_3 KP_{1,2}^T)$ 
is a strictly positive definite matrix on ${\cal V}_{1,2}$.
\end{lemma}

\begin{proofof}{Lemma \ref{6-13-thC}}
Let $\vec{u}_{1,2}$ and $\vec{u}_3$  be arbitrary vectors in 
${\cal V}_{1,2}$ and ${\cal V}_3$, respectively.
Schwartz inequality guarantees that
\begin{align}
\langle \vec{u}_{1,2}|  P_{1,2} K P_{1,2}^T | \vec{u}_{1,2}\rangle 
\langle \vec{u}_3 | P_3 K P_3^T | \vec{u}_3)\rangle
> 
\langle\vec{u}_3 |  P_3 K P_{1,2}^T | \vec{u}_{1,2}\rangle^2.
\end{align}
Choosing $|\vec{u}_3\rangle:= (P_3 K P_3^T)^{-1} P_3 K P_{1,2}^T | \vec{u}_{1,2}\rangle$, 
we have
\begin{align}
\langle \vec{u}_{1,2} | P_{1,2} K P_{1,2}^T | \vec{u}_{1,2}\rangle 
> 
\langle \vec{u}_{1,2}| (P_{1,2} K P_3^T)(P_3 KP_3^T)^{-1}
(P_3KP_{1,2}^T) 
|\vec{u}_{1,2}\rangle .
\end{align}
\end{proofof}

\begin{proofsof}{Lemma \ref{6-13-thL} and Theorem \ref{6-13-th}}

\noindent{\it Step 1 (Preparation):}\quad
In the following, we assume that
the matrix $\mathsf{H}_{0,\vec{\theta}_{2,3;o}}[\phi]$ is invertible.
Otherwise, we remove linear dependent generators among $g_1, \ldots, g_{l_1}$.
Notice that the linearly independence of $g_{l_1+1}, \ldots, g_{l_1+l_2+l_3}$ is guaranteed by the assumption C3.
Even if we make this change, 
the projective Fisher information matrix $\tilde{\mathsf{H}}_{\vec{\theta}_{2,3;o}} $ 
nor the estimator does not change.
So, we define the matrix $\mathsf{J}:=\mathsf{H}_{0,\vec{\theta}_{2,3;o}}[\phi]^{-1}$ 
on the whole space ${\cal V}_{1,2,3}$, and
the matrix $C$ as the square root of  
$P_{1,2} \mathsf{J} P_{1,2}^T -(P_{1,2} \mathsf{J} P_3^T) 
(P_3 \mathsf{J} P_3^T)^{-1} (P_3 \mathsf{J} P_{1,2}^T)$ 
on the subspace ${\cal V}_{1,2}$.
Lemma \ref{6-13-thC} guarantees that $C$ is a full rank matrix on ${\cal V}_{1,2}$.
Also, we define the $(l_1+l_2+l_3)\times (l_2+l_3)$ matrix $\mathsf{A}:=P_{1,2}^T A(\vec{\theta}_{2,3;o})$.
Now, we define the matrix $B$ on the direct sum space 
${\cal V}_{1,2,3}={\cal V}_{1,2}\oplus {\cal V}_3$ as
\begin{align}
B:= \left(
\begin{array}{cc}
C & 0 \\
(P_3 \mathsf{J} P_3^T)^{-1/2} (P_3 \mathsf{J} P_{1,2}^T) 
& (P_3 \mathsf{J} P_3^T)^{1/2} \\
\end{array}
\right).
\end{align}
So, the map $B$ maps the subspace ${\cal V}_{3}$ to itself.
Also, the matrix $B$ satisfies that
\begin{align}
P_3 BP_3^T &=(P_3 B P_3^T)^T \Label{eq-6-9-6} \\
B \mathsf{J}^{-1} B^T &=I \Label{eq-6-9-5} \\
P_{1,2} B &=
C P_{1,2} \Label{eq-6-9-3}\\
P_3^T P_3 B P_3 &=B P_3 \Label{eq-6-9-2}\\
B^T B &=\mathsf{J} \Label{eq-6-9-4} \\
B^T P_3^T P_3 B &= 
\left(
\begin{array}{cc}
(P_3\mathsf{J}P_{12}^T)^T  (P_3 \mathsf{J} P_3^T)^{-1} P_3 \mathsf{J}
(P_3\mathsf{J}P_{12}^T)
 & (P_3 \mathsf{J} P_{12}^T)^T \\
(P_3 \mathsf{J} P_{12}^T) & (P_3 \mathsf{J} P_3^T)
\end{array}
\right) \nonumber \\
&=\mathsf{J} P_3^T (P_3 \mathsf{J} P_3^T)^{-1} P_3 \mathsf{J}.\Label{eq-6-9-9} 
\end{align}
Also, we have
\begin{align}
&
P_{12}^T P_{12} B^T B P_{12}^T P_{12}
-
P_{12}^T P_{12} B^T P_3^T P_3  B P_{12}^T P_{12}\nonumber \\
= &
\left(
\begin{array}{cc}
C^2+ ((P_3 \mathsf{J} P_3^T)^{-1/2} (P_3 \mathsf{J} P_{1,2}^T) )^T
((P_3 \mathsf{J} P_3^T)^{-1/2} (P_3 \mathsf{J} P_{1,2}^T) )
 & 0 \\
0 & 0
\end{array}
\right) \nonumber \\
&-
\left(
\begin{array}{cc}
((P_3 \mathsf{J} P_3^T)^{-1/2} (P_3 \mathsf{J} P_{1,2}^T) )^T
((P_3 \mathsf{J} P_3^T)^{-1/2} (P_3 \mathsf{J} P_{1,2}^T) )
 & 0 \\
0 & 0
\end{array}
\right) \nonumber \\
=&
\left(
\begin{array}{cc}
C^2 & 0 \\
0 & 0
\end{array}
\right)
=
B^T P_{1,2}^T  P_{1,2} B \Label{2-12-A}.
\end{align}
Then,
\begin{align}
& (C A(\vec{\theta}_{2,3;o}))^{T} (C A(\vec{\theta}_{2,3;o}))
= (C P_{1,2} \mathsf{H}_{0,\vec{\theta}_{2,3}}[\phi] P_{2,3}^T)^{T} 
(C P_{1,2} \mathsf{H}_{0,\vec{\theta}_{2,3}}[\phi] P_{2,3}^T)
\nonumber \\
\stackrel{(a)}{=}
& (P_{1,2} B \mathsf{H}_{0,\vec{\theta}_{2,3}}[\phi] P_{2,3}^T)^{T} 
( P_{1,2} B \mathsf{H}_{0,\vec{\theta}_{2,3}}[\phi] P_{2,3}^T)
\nonumber\\
=&(\mathsf{H}_{0,\vec{\theta}_{2,3}}[\phi] P_{2,3}^T)^T 
(B^T  P_{1,2}^T  P_{1,2} B) (\mathsf{H}_{0,\vec{\theta}_{2,3}}[\phi] P_{2,3}^T )
\Label{01-20-e2} \\
\stackrel{(b)}{=} &
(\mathsf{H}_{0,\vec{\theta}_{2,3}}[\phi] P_{2,3}^T)^T 
(P_{12}^T P_{12} B^T B P_{12}^T P_{12}
-
P_{12}^T P_{12} B^T P_3^T P_3  B P_{12}^T P_{12} )
 \mathsf{H}_{0,\vec{\theta}_{2,3}}[\phi] P_{2,3}^T  \nonumber\\
= &
(P_{12}^T P_{12} \mathsf{H}_{0,\vec{\theta}_{2,3}}[\phi] P_{2,3}^T)^T 
( B^T B - B^T P_3^T P_3  B  )
P_{12}^T P_{12} \mathsf{H}_{0,\vec{\theta}_{2,3}}[\phi] P_{2,3}^T  \nonumber\\
\stackrel{(c)}{=}  &
\mathsf{A}^T
(\mathsf{J} -\mathsf{J} P_3^T (P_3\mathsf{J} P_3^T)^{-1}P_3 \mathsf{J})  
\mathsf{A}
 \nonumber\\
=& \tilde{\mathsf{H}}_{0,\vec{\theta}_{2,3}} 
  ,\Label{01-20-e}
\end{align}
where 
the equation $(a)$ follows from \eqref{eq-6-9-3},
the equation $(b)$ does  from \eqref{2-12-A},
and
the equation $(c)$ does  from \eqref{eq-6-9-4} and \eqref{eq-6-9-9}.

\noindent{\it Step 2 (Proof of Lemma \ref{6-13-thL}):}\quad
Since $C$ is is a full rank matrix on ${\cal V}_{1,2}$,
\eqref{01-20-e} guarantees that
the rank of $\mathsf{A}$ is the same as that of 
$\tilde{\mathsf{H}}_{0,\vec{\theta}_{2,3}} $.
 So, we obtain the desired statement.

\noindent{\it Step 3 (Proof of Theorem \ref{6-13-th}):}\quad
We show the statement with three steps.

\noindent{\it Step 3-1:}\quad
For this purpose, we introduce another parametrization of a transition matrix.
For $\eta \in {\cal V}_{1,2,3}$, we define the transition matrix $W^m_{\vec{\eta}}$
as $W^m_{\vec{\eta}} = W_{\vec{\theta}}$ with the condition $\vec{\eta}=\vec{\eta}(\vec{\theta})$.
Also, we introduce another parametrization
\begin{align}
\vec{\xi}(\vec{\eta}):=
\sqrt{n} B (\vec{\eta}-\vec{\eta}(0,\vec{\theta}_{2,3;o} ) ).
\end{align}
Thus, when $\vec{\eta}=\vec{\eta}(\vec{\theta}) $ and $\vec{\eta}'$ is close to $\vec{\eta}$,
\eqref{eq-6-9-5} implies that
\begin{align}
 2n  D (W^m_{\vec{\eta}}\| W^m_{\vec{\eta}'})
= & n (\vec{\eta}-\vec{\eta}')^T\cdot 
\mathsf{J} \cdot (\vec{\eta}-\vec{\eta}')
+ o(n \|\vec{\eta}-\vec{\eta}'\|^2) \nonumber \\
= &
\|\vec{\xi}(\vec{\eta})-\vec{\xi}(\vec{\eta}')\|^2
+o(\|\vec{\xi}(\vec{\eta})-\vec{\xi}(\vec{\eta}')\|^2).
\Label{2-8-1}
\end{align}

Also, we introduce two vectors
\begin{align}
\vec{\xi}_{1,2} :=& P_{1,2} \vec{\xi}(\vec{Y}_{1,2}^n,0) 
=\sqrt{n} P_{1,2}B (\vec{Y}_{1,2}^n-\vec{\eta}_{1,2}(0,\vec{\theta}_{2,3;o}),0) \nonumber\\
=& \sqrt{n} C (\vec{Y}_{1,2}^n-\vec{\eta}_{1,2}(0,\vec{\theta}_{2,3;o})),
\\
\vec{\xi}_3' := & P_{3} \vec{\xi}(\vec{Y}_{1,2}^n,\vec{\eta}_3') \nonumber \\
= &
\sqrt{n}P_{3}B (\vec{Y}_{1,2}^n-\vec{\eta}_{1,2}(0,\vec{\theta}_{2,3;o}),0)
+ 
\sqrt{n}P_{3} B(0,\vec{\eta}_3'-\vec{\eta}_{3}(0,\vec{\theta}_{2,3;o})) .
\end{align}
Since the map $B$ maps the subspace ${\cal V}_{3}$ to itself,
the vector $\vec{\xi}_3'$ also belongs to ${\cal V}_3$.
Thus,
\begin{align}
\vec{\xi}(\vec{Y}_{1,2}^n,\vec{\eta}_3')
=\vec{\xi}_{1,2}+\vec{\xi}_{3}'.
\end{align}

Further, we divide the subspace ${\cal V}_{1,2}$ into two orthogonal spaces ${\cal V}_4$ and ${\cal V}_5$ such that
${\cal V}_4$ is the image of 
\begin{align}
P_{1,2} B \mathsf{H}_{0,\vec{\theta}_{2,3;o}}[\phi] P_{2,3}^T
=P_{1,2} B P_{1,2}^T P_{1,2}\mathsf{H}_{0,\vec{\theta}_{2,3;o}}[\phi] P_{2,3}^T
=P_{1,2} B \mathsf{A}.\Label{MMG1}
\end{align}
We denote the projection to ${\cal V}_j$ by 
the matrix $P_j$ from ${\cal V}_{1,2}$ to ${\cal V}_j$ for $j=4,5$.
Hence, 
\begin{align}
&P_{2,3} \mathsf{H}_{0,\vec{\theta}_{2,3;o}}[\phi]  B^T
 P_{1,2}^T P_4^T
P_4 P_{1,2} B  \mathsf{H}_{0,\vec{\theta}_{2,3;o}}[\phi] P_{2,3}^T
 \nonumber \\
=&
P_{2,3} \mathsf{H}_{0,\vec{\theta}_{2,3;o}}[\phi]  B^T
 P_{1,2}^T P_{1,2}  B  \mathsf{H}_{0,\vec{\theta}_{2,3;o}}[\phi] P_{2,3}^T.
\Label{MMG2}
\end{align}
Due to the condition {\bf C2}, the dimension of ${\cal V}_4$ is $l_2+l_3$.
So, the dimension of ${\cal V}_5$ is $l_1-l_3$.
Since the input and output spaces of $P_4 P_{1,2} B  \mathsf{H}_{0,\vec{\theta}_{2,3;o}}[\phi] P_{2,3}^T$ have the same dimension 
$l_2+l_3$ and it is surjective, we can consider its inverse matrix.

\noindent{\it Step 3-2:}\quad
We show 
\begin{align}
\sqrt{n}(\hat{\theta}_{2,3}(\vec{Y}_{1,2}^n)-\vec{\theta}_{2,3;o}) \cong
(P_4 P_{1,2} B  \mathsf{H}_{0,\vec{\theta}_{2,3;o}}[\phi] P_{2,3}^T)^{-1} 
P_4 \vec{\xi}_{1,2}.
\Label{2-8-6}
\end{align}

Using 
\begin{align}
\zeta_{2,3}':=\sqrt{n} (\vec{\theta}_{2,3}'-\vec{\theta}_{2,3;o}),
\Label{2-8-8}
\end{align}
we have
\begin{align}
& 2n D(W_{\vec{\theta}(\vec{Y}_{1,2}^n,\vec{\eta}_3')}\|W_{0,\vec{\theta}_{2,3}'})
\cong 
\|\vec{\xi}(\vec{Y}_{1,2}^n,\vec{\eta}_3')-\vec{\xi}(\vec{\eta}(0,\vec{\theta}_{2,3}'))\|^2 \nonumber \\
=&
\|
(\vec{\xi}_{1,2},\vec{\xi}_{3}') - B \mathsf{H}_{0,\vec{\theta}_{2,3;o}}[\phi] \zeta_{2,3}'
\|^2 \nonumber \\
=&
\|\vec{\xi}_{1,2}-P_{12} B \mathsf{H}_{0,\vec{\theta}_{2,3;o}}[\phi] \zeta_{2,3}'\|^2
+\|\vec{\xi}_3'
-P_3 B \mathsf{H}_{0,\vec{\theta}_{2,3;o}}[\phi] \zeta_{2,3}'
\|^2 . \Label{2-8-4}
\end{align}
So, the minimization with respect to $\vec{\eta}_3$ is converted to 
that with respect to $\vec{\xi}_3'$.
That is, the minimum is realized when 
$\vec{\xi}_3'=P_3 B \mathsf{H}_{0,\vec{\theta}_{2,3;o}}[\phi] \zeta_{2,3}'$.

Next, we consider the minimization of the first term
$\|\vec{\xi}_{1,2}-P_{12} B \mathsf{H}_{0,\vec{\theta}_{2,3;o}}[\phi] \zeta_{2,3}'\|^2$. 
The definitions of $P_4$ and $P_5$ yield that
\begin{align}
&\| \vec{\xi}_{1,2}-P_{1,2} B \mathsf{H}_{0,\vec{\theta}_{2,3;o}}[\phi] 
P_{2,3}^T 
\vec{\zeta}_{2,3}'\|^2 
\nonumber \\
=&
\| P_5 \vec{\xi}_{1,2}\|_2
+\| P_4 \vec{\xi}_{1,2}- P_4 P_{1,2} B \mathsf{H}_{0,\vec{\theta}_{2,3;o}}[\phi] P_{2,3}^T \vec{\zeta}_{2,3}'\|^2  \nonumber \\
=&
\| P_5 \vec{\xi}_{1,2}\|_2\nonumber \\
&+\big\| (P_4 P_{1,2} B  \mathsf{H}_{0,\vec{\theta}_{2,3;o}}[\phi] P_{2,3}^T ) 
((P_4 P_{1,2} B  \mathsf{H}_{0,\vec{\theta}_{2,3;o}}[\phi] P_{2,3}^T)^{-1} 
P_4 \vec{\xi}_{1,2}
- \vec{\zeta}_{2,3}')\big\|^2.
\Label{6-13-1}
\end{align}
So, 
\begin{align}
\hat{\zeta}_{2,3}
:= &
\argmin_{\vec{\zeta}_{2,3}'} 
\big\| \vec{\xi}_{1,2}-P_{1,2} B \mathsf{H}_{0,\vec{\theta}_{2,3;o}}[\phi] 
P_{2,3}^T \vec{\zeta}_{2,3}'\big\|^2 \nonumber \\
=&
(P_4 P_{1,2} B  \mathsf{H}_{0,\vec{\theta}_{2,3;o}}[\phi] P_{2,3}^T)^{-1} P_4 \vec{\xi}_{1,2}.
\Label{2-8-3}
\end{align}
Since $
\hat{\theta}_{2,3}(\vec{Y}_{1,2}^n)=
\argmin_{ \vec{\theta}_{2,3}' \in {\cal V}_{2,3}}
\min_{\vec{\eta}_3' \in {\cal V}_3} D(W_{\vec{\theta}(\vec{Y}_{1,2}^n,\vec{\eta}_3)}\|W_{0,\vec{\theta}_{2,3}'})$, 
\eqref{2-8-4} and \eqref{2-8-8}, and
\eqref{2-8-3} guarantees that
\begin{align}
\sqrt{n}(\hat{\theta}_{2,3}(\vec{Y}_{1,2}^n)-\vec{\theta}_{2,3;o})
\cong 
\hat{\zeta}_{2,3},
\end{align}
which implies \eqref{2-8-6} because of \eqref{2-8-3}.

\noindent{\it Step 3-3:}\quad
We show the statement by using \eqref{2-8-6}.
Proposition \ref{P2-9-1} guarantees that the random variable $\sqrt{n}(\vec{Y}_{1,2}^n- \vec{\eta}_{1,2}(0,\vec{\theta}_{2,3;o}))$
asymptotically obeys the Gaussian distribution with the covariance matrix 
$ P_{1,2} \mathsf{H}_{0,\vec{\theta}_{2,3;o}}[\phi] P_{1,2}^T$.
Hence, due to \eqref{eq-6-9-5} and \eqref{eq-6-9-3}, 
the random variable 
$\vec{\xi}_{1,2}$ 
asymptotically obeys the Gaussian distribution with the covariance matrix 
\begin{align}
P_{1,2} B P_{1,2} \mathsf{H}_{0,\vec{\theta}_{2,o}}[\phi] P_{1,2}^T B^T P_{1,2}^T
= P_{1,2} B  \mathsf{J}^{-1} B^T P_{1,2}^T
=P_{1,2} I P_{1,2}^T= P_{1,2} P_{1,2}^T.
\Label{2-8-8B}
\end{align}

Thus, the equation \eqref{2-8-6} guarantees that
the random variable 
$\sqrt{n}(\hat{\theta}_2(\vec{Y}_1^n)-\vec{\theta}_{2,o}) $
asymptotically obeys the Gaussian distribution whose covariance matrix is
\begin{align}
& (P_4 P_{1,2} B  \mathsf{H}_{0,\vec{\theta}_{2,3;o}}[\phi] P_{2,3}^T)^{-1}
((P_4 P_{1,2} B  \mathsf{H}_{0,\vec{\theta}_{2,3;o}}[\phi] P_{2,3}^T)^T)^{-1} \nonumber \\
=&
( (P_4 P_{1,2} B  \mathsf{H}_{0,\vec{\theta}_{2,3;o}}[\phi] P_{2,3}^T)^T
P_4 P_{1,2} B  \mathsf{H}_{0,\vec{\theta}_{2,3;o}}[\phi] P_{2,3}^T
)^{-1} \nonumber \\
=&
( 
P_{2,3} \mathsf{H}_{0,\vec{\theta}_{2,3;o}}[\phi]  B^T
 P_{1,2}^T P_4^T
P_4 P_{1,2} B  \mathsf{H}_{0,\vec{\theta}_{2,3;o}}[\phi] P_{2,3}^T
)^{-1} \nonumber \\
\stackrel{(a)}{=}&
( 
P_{2,3} \mathsf{H}_{0,\vec{\theta}_{2,3;o}}[\phi]  B^T
 P_{1,2}^T P_{1,2}  B  \mathsf{H}_{0,\vec{\theta}_{2,3;o}}[\phi] P_{2,3}^T
)^{-1} \nonumber \\
=&
( ( \mathsf{H}_{0,\vec{\theta}_{2,3;o}}[\phi]  P_{2,3}^T)^T
(B^T P_{1,2}^T P_{1,2}  B ) (\mathsf{H}_{0,\vec{\theta}_{2,3;o}}[\phi] P_{2,3}^T)
)^{-1} 
\stackrel{(b)}{=}
(\tilde{\mathsf{H}}_{0,\vec{\theta}_{2,3}} )^{-1}  ,
\end{align}
where 
the equation $(a)$ follows from \eqref{MMG2}, 
and
the equation $(b)$ does from the equation between \eqref{01-20-e2} and \eqref{01-20-e}.
Thus, we obtain the statement.

\end{proofsof}

\begin{proofof}{Theorem \ref{6-13-thB}}
Now, we prove Theorem \ref{6-13-thB} by using the notations in the proof of Theorem \ref{6-13-th}.
In this setting,
the dimension of ${\cal V}_4$ is $l_2+l[\Theta_{2,3}]$.
So, the dimension of ${\cal V}_5$ is $l_1-l[\Theta_{2,3}]$.
The relations \eqref{2-8-4}, \eqref{6-13-1} and \eqref{2-8-3} guarantee that
\begin{align}
2n \min_{\vec{\theta}_{2,3}' \in \Theta_{2,3}}
\min_{\vec{\eta}_3' \in {\cal V}_3}
D\big(W_{\vec{\theta}(\vec{Y}_{1,2}^n,\vec{\eta}_3')}\big\|W_{0,\vec{\theta}_{2,3}'}\big)
\cong \| P_5 \vec{\xi}_{1,2}\|_2.
\end{align}
Due to \eqref{2-8-8B},
$\vec{\xi}_{1,2}$ is subject to the standard Gaussian distribution.
Since $P_5$ is the projection to the $l_1-l[\Theta_{2,3}]$-dimensional space 
${\cal V}_5$,
the random variable $ \| P_5 \vec{\xi}_{1,2}\|_2$
is subject to the $\chi^2$-distribution with degree $l_1-l[\Theta_{2,3}]$.
Hence, we obtain the desired statement.
\end{proofof}

\section{Perron-Frobenius theorem}\Label{PFTH}
Since we employ Perron-Frobenius theorem in this paper, we review it.
When any component of a matrix $W$ is a non-negative real number,
$W$ is called a positive matrix.
\begin{proposition}[\protect{\cite[Theorem 3.1.]{DZ}\cite{Sen}}]
Let $W = (W_{i,j})$ be an irreducible positive matrix. 
(1) There uniquely exists a positive real number $\lambda>0$ such that 
$\lambda$ is an eigenvalue of $W$ and any other eigenvalue $a$ (possibly, complex) is strictly smaller than r in absolute value, $|a| < \lambda$. 
(2) There exist eigenvectors $v$ and $v'$ of $W$ and $W^T$ corresponding to the eigenvalue $\lambda$
that have strictly positive components.
(3) In addition, these eigenvectors are unique up to a constant multiple.
\end{proposition}

The above eigenvalue $\lambda$ is called the Perron-Frobenius eigenvalue of $W$,
and the above eigenvector of $W$ is called the Perron-Frobenius eigenvector of $W$.

\end{document}